\documentclass[a4paper,11pt, reqno]{amsart}
\usepackage{latexsym}
\usepackage{color}
\usepackage{amssymb,amsfonts,amsmath,mathrsfs}
\newtheorem{theorem}{Theorem}
\newtheorem{lemma}[theorem]{Lemma}
\newtheorem{corollary}[theorem]{Corollary}

\newtheorem{proposition}[theorem]{Proposition}

\newtheorem{remark}[theorem]{Remark}

\newcommand{\kommentar}[1]{}
\newcommand{\hstar}{\mathcal{H}_{2g+1}^{*}}

\numberwithin{theorem}{section} \numberwithin{equation}{section}

\allowdisplaybreaks

\begin{document}

\title[Quadratic twists of elliptic curve $L$--functions over function fields]{Moments of quadratic twists of elliptic curve $L$--functions over function fields}
\author{H. M. Bui, Alexandra Florea, J. P. Keating and E. Roditty-Gershon}
\address{School of Mathematics, University of Manchester, Manchester M13 9PL, UK}
\email{hung.bui@manchester.ac.uk}
\address{Department of Mathematics, Columbia University, New York NY 10027, USA}
\email{aflorea@math.columbia.edu}
\address{School of Mathematics, University of Bristol, Bristol BS8 1TW, UK}
\email{j.p.keating@bristol.ac.uk}
\address{Department of Applied Mathematics, H.I.T. - Holon Institute of Technology, Holon 5810201, Israel}
\email{rodittye@gmail.com}
\date{}

\maketitle
\begin{abstract}
We calculate the first and second moments of $L$--functions in the family of quadratic twists of a fixed elliptic curve $E$ over $\mathbb{F}_q[x]$, asymptotically in the limit as the degree of the twists tends to infinity. We also compute moments involving derivatives of $L$--functions over quadratic twists, enabling us to deduce lower bounds on the correlations between the analytic ranks of the twists of two distinct curves.
\end{abstract}

\section{Introduction and statement of results}
The values of $L$-functions at the central point of the critical strip have been the subject of considerable interest in recent years. One way to study these central values is by considering moments in families of $L$--functions. There are now precise conjectured asymptotic formulas for such moments motivated by analogies with Random Matrix Theory \cite{snaith2, snaith}.
More precise asymptotic formulas containing lower order terms were conjectured in \cite{ conrey2, Hoff, conrey}. In the case of the Riemann zeta-function, the analogue of these conjectures is now relatively well understood in terms of correlations of the divisor function \cite{conreykeating1, conreykeating2, conreykeating3, conreykeating4, conreykeating5}.  The moments of other degree-one $L$-functions have also been investigated intensively.  It remains a challenge to extend these calculations to $L$-functions of degree two and higher.

The first moment of the family of quadratic twists of a fixed modular form was studied in \cite{BFH, IW, MM}. Questions related to the nonvanishing of $L$--functions in this family were considered in \cite{BFH, MM, BFH2}. For example, it is shown independently in \cite{MM} and \cite{BFH2}, using different techniques, that there are infinitely many fundamental discriminants $d<0$ such that for a fixed elliptic curve with root number equal to $1$, its twist by $d$ has analytic rank equal to $1$. 

The second moment of the family was considered by Soundararajan and Young in \cite{SoundYoung}. Unconditionally, they obtained a lower bound for the second moment which matches the asymptotic formula conjectured by Keating and Snaith \cite{snaith} and assuming the Generalized Riemann Hypothesis (GRH) they established the conjectured formula. Using similar ideas, again under GRH, Petrow \cite{petrow} obtained several asymptotic formulas for moments of derivatives of these $GL(2) \, \,  L$--functions when the sign of the functional equation is $-1$. 

While no asymptotic formulas for moments larger than the second are known for this family, there are lower and upper bounds of the right order of magnitude. Rudnick and Soundararajan \cite{RudnickSound1, RudnickSound2} established unconditional lower bounds for all moments larger than the first, and, assuming GRH, the work of Soundararajan \cite{SoundUB} and its refinement by Harper \cite{Harper} produced upper bounds of the conjectured order of magnitude. In \cite{rad} Radziwi\l\l\ and Soundararajan proved upper bounds for moments below the first in this family of $L$--functions. Their techniques also allow them to obtain a one-sided central limit for the distribution of the logarithm of these central $L$--values. Their result supports a conjecture by Keating and Snaith \cite{snaith2} which can be viewed as the analogue of Selberg's central limit theorem for the distribution of $\log|\zeta(1/2+it)|$.

In this paper we study several moment problems of comparable difficulty to the moment computation of Soundararajan and Young in the function field family of quadratic twists of an elliptic curve. Since we are working over function fields, the results we obtain are unconditional. 

Recently there has been a good deal of work on computing moments of $L$--functions in the function field setting. Andrade and Keating \cite{andradekeating} obtained an asymptotic formula for the first moment in the symplectic family of quadratic $L$--functions when the degree of the $L$--functions (which is a polynomial in this case) goes to infinity and the size of the finite field is fixed (see also \cite{Ros} for a similar result). A lower order term of size approximately the cube root of the main term was computed in \cite{fl1}. The second, third and fourth moments were computed in \cite{fl2,fl4} (see also \cite{diac}). We note that the asymptotic formula for the fourth moment does not have a power savings error term, but recovers several of the expected leading order terms in the conjectured formula \cite{ak_conj}. Obtaining an asymptotic formula with the leading order term for the fourth moment in the family of quadratic $L$--functions is comparable in difficulty to establishing an asymptotic formula for the second moment of $L$--functions of quadratic twists of an elliptic curve, and is one of the problems we consider in this paper. 

We note that for all of our results, we fix the size $q$ of the finite field we work in and let the degree of the $L$--functions go to infinity. If instead one fixes the degree and lets $q \to \infty$, then Katz and Sarnak \cite{katzsarnak} showed that the $L$--functions become equidistributed in the orthogonal group, and hence computing the various moments reduces to computing several random matrix integrals (see for example \cite{snaith2}). In the case of elliptic curve $L$-functions, the relevant equidistribution results were established in \cite{ger}.

%%%%%%%%%%%%%%%%%%%%%%%%%%%%%
To state our results we first need some notation.
Fix a prime power $q$ with $(q,6)=1$ and $q \equiv 1 \pmod 4$. Let $K= \mathbb{F}_q(t)$ be the rational function field and $\mathcal{O}_K = \mathbb{F}_q[t]$. %For a monic irreducible polynomial, let $\mathbb{F}_{\pi}$ denote its residue field, while for $\pi = \infty$ we identify $\mathbb{F}_{\pi}$ with $\mathbb{F}_q[u]/u$, where $u=1/t$. 
Let $E/K$ be an elliptic curve defined by $y^2=x^3+ax+b$, with $a, b \in \mathcal{O}_K$ and discriminant $\Delta = 4a^3+27b^2$ such that $\deg_t ( \Delta)$ is minimal. %Let $M$ denote the finite set of places with multiplicative reduction, $A$ denotes the finite set of places with additive reduction and $U$ stands for the set of places with good reduction. Let $M = M^{+} \cup M^{-},$ where $\mplus$ denotes the subset of places with split multiplicative reduction, and $\mminus$ is the subset of places with non-split multiplicative reduction. 

%For each place $\pi$ with good reduction in $U$, let $\mathbb{F}_{\pi^m}$ be the unique extension of degree $m$ of $\fpi$, and let $E(\mathbb{F}_{\pi^m})$ denote the set of $\mathbb{F}_{\pi^m}$-rational points on the curve reduced $\bmod \pi$. Let $$a_{\pi^m} = q^{m \deg( \pi)} +1 - | E(\mathbb{F}_{\pi^m}) |.$$
%The normalized $L$--function associated to the elliptic curve $E/K$ has the following Euler product when $\Re(s) > 1$:
%$$ \mathcal{L}(E,u) = \prod_{P \nmid \Delta} \Big( 1 - \frac{a_{\pi} u^{\deg(\pi)}}{q^{\deg(\pi)/2}} +  u^{2 \deg(\pi)} \Big)^{-1} \prod_{P \in \mplus} \Big(1 - \frac{u^{\deg(P)}}{\sqrt{|P|}} \Big)^{-1}  \prod_{P \in \mminus} \Big(1 + \frac{u^{\deg(P)}}{\sqrt{|P|}} \Big)^{-1}.$$
The normalized $L$--function associated to the elliptic curve $E/K$ has the following Euler product and Dirichlet series, which converge for $\Re(s)>1$,
\begin{align*}
L(E,s):=\mathcal{L}(E,u)& = \sum_{f \in \mathcal{M}} \lambda(f) u^{\deg(f)} \\
&= \prod_{P | \Delta} \Big(1- \lambda(P) u^{\deg(P)}\Big)^{-1} \prod_{P \nmid \Delta} \Big(1- \lambda(P) u^{\deg(P)} + u^{2 \deg(P)}\Big)^{-1},
\end{align*}
where we set $u:=q^{-s}$, and $\mathcal{M}$ denotes the set of monic polynomials over $\mathbb{F}_q[t]$. The $L$--function is a polynomial in $u$ with integer coefficients of degree
\begin{equation}
 \mathfrak{n} := \deg\big( \mathcal{L}(E,u)\big) = \deg(M) + 2 \deg(A) -4,\label{degree}
 \end{equation} where for simplicity we denote by $M$ the product of the finite primes with multiplicative reduction and by $A$ the product of the finite primes with additive reduction. Moreover, the $L$--function satisfies the functional equation; namely, there exists $\epsilon(E) \in \{ \pm 1 \}$ such that
$$ \mathcal{L}(E,u) = \epsilon(E) (\sqrt{q} u)^{\mathfrak{n}} \mathcal{L} \Big( E, \frac{1}{qu}\,\Big).$$
For a more precise formula for the sign of the functional equation, see Lemma $2.3$ in \cite{hall}.
Now for $D \in \mathcal{O}_K$ with $D$ square-free, monic of odd degree and $(D, \Delta)=1$, we consider the twisted elliptic curve $E \otimes \chi_D /K$ with the affine model $y^2 = x^3+ D^2 ax+ D^3b$. Then the $L$--function corresponding to the twisted elliptic curve has the following Dirichlet series and Euler product
\begin{align*} 
\mathcal{L}(E \otimes \chi_D,u) &= \sum_{f \in \mathcal{M}} \lambda(f) \chi_D(f) u^{\deg(f)}\\
&= \prod_{P | \Delta}\Big (1- \lambda(P) \chi_D(P) u^{\deg(P)}\Big)^{-1} \prod_{P \nmid \Delta D}\Big (1- \lambda(P) \chi_D (P) u^{\deg(P)} +u^{2 \deg(P)}\Big)^{-1}.\end{align*}
The new $L$--function is a polynomial of degree $\big(\mathfrak{n}+ 2 \deg(D)\big)$ and satisfies the functional equation%\footnote{\hcom{I avoid using the notation $\mathfrak{n}_D$ as later we will have two elliptic curves. So $\mathfrak{n}$ is reserved for $\deg(\mathcal{L}(E,u))$, $\mathfrak{n}_1$ is reserved for $\deg(\mathcal{L}(E_1,u))$, etc.}}
\begin{equation}\label{tfe}
\mathcal{L}(  E \otimes \chi_D, u ) = \epsilon\, (\sqrt{q} u)^{\mathfrak{n}+ 2 \deg(D)} \mathcal{L}\Big (  E \otimes \chi_D, \frac{1}{qu}\,\Big),
\end{equation}
where
$$\epsilon= \epsilon(E \otimes \chi_D) = \epsilon_{\deg(D)} \epsilon(E)\chi_D(M).$$ Here $\epsilon_{\deg( D)} \in \{ \pm 1 \}$ is an integer which only depends on the degree of $D$ (see Proposition $4.3$ in \cite{hall}).

Let $\mathcal{H}_{2g+1}^{*}$ denote the set of monic, square free polynomials of degree $(2g+1)$ coprime to $\Delta$. Our first two theorems concern the first and second moments of $L (E \otimes \chi_D, 1/2)$.

\begin{theorem}\label{L^1}
Unless $\epsilon_{2g+1}\epsilon(E)=-1$ and $M=1$, we have
\begin{align*}
\frac{1}{|\mathcal{H}_{2g+1}^{*}|}\sum_{D\in\mathcal{H}_{2g+1}^{*}}\, L (E \otimes \chi_D, \tfrac{1}{2}) =c_{1}(M)\,+O_\varepsilon\big(q^{-g+\epsilon g}\big),
\end{align*}
where the value $c_1(M)$ is defined in \eqref{c1} and \eqref{def_bm}. In particular, the constant $c_1(M)\ne0$ in this case and we obtain an asymptotic formula.
\end{theorem}

\begin{theorem}\label{L^2}
Unless $\epsilon_{2g+1}\epsilon(E)=-1$ and $M=1$, we have
\begin{align*}
\frac{1}{|\mathcal{H}_{2g+1}^{*}|}\sum_{D\in\mathcal{H}_{2g+1}^{*}}\, L (E \otimes \chi_D, \tfrac{1}{2})^2 =c_{2}(M)L\big(\emph{Sym}^2E,1\big)^3\,g+O_\varepsilon\big(g^{1/2+\varepsilon}\big),
\end{align*}
where the value $c_2(M)$ is defined in \eqref{C2}, \eqref{Euler1} and \eqref{Euler2}. In particular, the constant $c_2(M)\ne0$ in this case and we obtain an asymptotic formula.
\end{theorem}

Note that the Theorem above is the function field analogue of Theorem $1.2$ in \cite{SoundYoung}. Considering the smoothed second moment, Soundararajan and Young obtain an error term of size $(\log X)^{3/4+\epsilon}$, which would translate to $g^{3/4+\epsilon}$ in the function field setting. Using slightly different techniques, Petrow states in \cite{petrow} that the error term could be improved to $(\log X)^{1/2+\epsilon}$ which is of the same quality we obtain in the result above.   

Our Theorem \ref{L^2} should also be compared to the asymptotic formula for the fourth moment of quadratic $L$--functions over function fields in \cite{fl4}. We remark that for the symplectic family of quadratic $L$--functions, one can obtain lower order terms in the asymptotic formula by using an inductive argument and then obtaining upper bounds for moments of $L$--functions evaluated at points far from the critical point. The fact that one can compute a few lower order terms can be explained by the gap between  powers of $g$ coming from evaluating moments at the critical point versus evaluating moments far from the critical point. When computing the fourth moment of quadratic $L$--functions close to the central point, one expects to obtain a power of $g^{10}$. As we move away from the central point, the family starts to behave like a family with unitary symmetry and one expects an upper bound of the magnitude $g^4$. The difference in powers of $g$ gives one room to use a repetitive argument to rigorously compute lower order terms down to $g^4$. In the case of the orthogonal family we consider in this paper, note that the main term in Theorem \ref{L^2} is of size $g$, and the error term has size $g^{1/2}$ coming from obtaining an upper bound for the second moment evaluated at a point far from the central point. The small difference between these powers of $g$ does not give us enough room to compute a lower order term in this case.

We can also study the moment of the product of the quadratic twists of two elliptic curve $L$-functions. Let $E_1$ and $E_2$ be two elliptic curves over $K$. Let $\Delta= \Delta_1 \Delta_2$, where, for $i=1,2$, $\Delta_i$ denotes the discriminant of $E_i$. Let $M_i$ denote the product of the finite primes with multiplicative reduction and $$\epsilon_i = \epsilon_{\deg(D)} \epsilon(E_i)\chi_D(M_i).$$ Define
\[
\epsilon_i^+:=\frac{1+\epsilon_i}{2}\qquad\text{and}\qquad \epsilon_i^-:=\frac{1-\epsilon_i}{2}.
\]

\begin{theorem}\label{L'L}
Unless $\epsilon_{2g+1}\epsilon(E_1)=-1$ and $M_1=1$, or $\epsilon_{2g+1}\epsilon(E_2)=1$ and $M_2=1$, or $\epsilon(E_1)=\epsilon(E_2)$ and $M_1=M_2$, we have
\begin{align*}
&\frac{1}{|\mathcal{H}_{2g+1}^{*}|}\sum_{D\in\mathcal{H}_{2g+1}^{*}}\, \epsilon_2^-\,L (E_1 \otimes \chi_D, \tfrac{1}{2})L' (E_2 \otimes \chi_D, \tfrac{1}{2}) \\
&\qquad\qquad=c_{3}(M_1,M_2)L\big(\emph{Sym}^2E_1,1\big)L\big(\emph{Sym}^2E_2,1\big)L\big(E_1\otimes E_2,1\big)\,g+O_\varepsilon\big(g^{1/2+\varepsilon}\big),
\end{align*}
where the value $c_3(M_1,M_2)$ is defined in \eqref{C3}, \eqref{Euler1} and \eqref{Euler3}. In particular, the constant $c_3(M_1,M_2)\ne0$ in this case and we obtain an asymptotic formula.
\end{theorem}

\begin{theorem}\label{L'^2}
Unless $\epsilon_{2g+1}\epsilon(E_1)=1$ and $M_1=1$, or $\epsilon_{2g+1}\epsilon(E_2)=1$ and $M_2=1$, or $\epsilon(E_1)=-\epsilon(E_2)$ and $M_1=M_2$, we have
\begin{align*}
&\frac{1}{|\mathcal{H}_{2g+1}^{*}|}\sum_{D\in\mathcal{H}_{2g+1}^{*}}\, \epsilon_1^-\epsilon_2^-\,L' (E_1 \otimes \chi_D, \tfrac{1}{2})L'(E_2 \otimes \chi_D, \tfrac{1}{2})\\
&\qquad\qquad =c_{4}(M_1,M_2)L\big(\emph{Sym}^2E_1,1\big)L\big(\emph{Sym}^2E_2,1\big)L\big(E_1\otimes E_2,1\big)\,g^2+O_\varepsilon\big(g^{1+\varepsilon}\big),
\end{align*}
where the value $c_4(M_1,M_2)$ is defined in \eqref{C4}, \eqref{Euler1} and \eqref{Euler3}. In particular, the constant $c_4(M_1,M_2)\ne0$ in this case and we obtain an asymptotic formula.
\end{theorem}
Note that this result is the analogue of Theorem $2.2$ in \cite{petrow}.
An interesting problem would be to compute the average of $L(E_1 \otimes \chi_D,1/2) L(E_2\otimes \chi_D,1/2)$ for distinct elliptic curves $E_1$ and $E_2$. This would have applications to the question of simultaneous non-vanishing of $L(E_1 \otimes \chi_D,1/2) $ and $L(E_2 \otimes \chi_D,1/2)$. However our techniques do not allow us to obtain such an asymptotic formula.

Define the analytic rank of a quadratic twist of an elliptic curve $L$-function $L(E\otimes\chi_D,s)$ by
\[
r_{E\otimes\chi_D}:=\text{ord}_{s=1/2}L(E\otimes\chi_D,s).
\]
Combining the upper bounds for moments of elliptic curve $L$-functions (see Section \ref{upperbounds}) with Theorems \ref{L'L} and \ref{L'^2} leads to the following corollary.

\begin{corollary}
\label{cor1}
Unless $\epsilon_{2g+1}\epsilon(E_1)=-1$ and $M_1=1$, or $\epsilon_{2g+1}\epsilon(E_2)=1$ and $M_2=1$, or $\epsilon(E_1)=\epsilon(E_2)$ and $M_1=M_2$, we have
\begin{align*}
\#\Big\{D\in\hstar: r_{E_1\otimes\chi_D}=0, r_{E_2\otimes\chi_D}=1 \Big\}\gg_\varepsilon\frac{q^{2g}}{g^{6+\varepsilon}}
\end{align*}
as $g\rightarrow\infty$. Also, unless $\epsilon_{2g+1}\epsilon(E_1)=1$ and $M_1=1$, or $\epsilon_{2g+1}\epsilon(E_2)=1$ and $M_2=1$, or $\epsilon(E_1)=-\epsilon(E_2)$ and $M_1=M_2$, we have
\begin{align*}
\#\Big\{D\in\hstar: r_{E_1\otimes\chi_D}=r_{E_2\otimes\chi_D}=1 \Big\}\gg_\varepsilon\frac{q^{2g}}{g^{6+\varepsilon}}
\end{align*}
as $g\rightarrow\infty$.
\end{corollary}

As far as we are aware, Corollary \ref{cor1} is the first result in literature where explicit lower bounds concerning the correlations between the ranks of two twisted elliptic curves are obtained.
Following Harper's argument for the upper bounds for moments of $L$-functions \cite{Harper}, one may remove the exponents $\varepsilon$ in Corollary \ref{cor1}. We fail to obtain positive proportions in the above results because we are not able to use a mollifier. Note that the results of Heath-Brown \cite{H-B} adapted to the function field setting do not lead to positive proportions either.

\medskip
{\bf Acknowledgements.} A. Florea gratefully acknowledges the support of an NSF Postdoctoral Fellowship during part of the research which led to this paper. J.P. Keating is supported by a Royal Society Wolfson Research Merit Award, EPSRC Programme Grant EP/K034383/1 LMF: $L$-Functions and Modular Forms, and by ERC Advanced Grant 740900 (LogCorRM). The authors would also like to thank Chantal David, Matilde Lalin and Zeev Rudnick for useful comments on the paper.

\section{Some useful lemmas}
In this section we will gather a few useful lemmas we will need throughout the paper. 

Recall that $q$ is a prime power with $q \equiv 1 \pmod 4$ and $(q,6)=1$. Let $\mathcal{M}$ denote the set of monic polynomials over $\mathbb{F}_q[t]$ and $\mathcal{H}$ be the set of monic, square-free polynomials. Let $\mathcal{M}_n$ denote the set of monic polynomials of degree $n$ over $\mathbb{F}_q[t]$ and $\mathcal{M}_{\leq n}$ be the set of monic polynomials of degree less than or equal to $n$. Let $\mathcal{H}_n$ denote the monic, square-free polynomials of degree $n$ and recall that $\mathcal{H}^{*}_n$ denotes the set of monic, square-free polynomials of degree $n$ coprime to $\Delta$. The norm of a polynomial $f$ is defined by $|f|= q^{\deg(f)}$.

We define the zeta-function as 
$$\zeta_q(s) = \sum_{f\in\mathcal{M}} \frac{1}{|f|^s}$$ for $\Re(s)>1$. By counting monic polynomials of a given degree, one can easily show that
$$\zeta_q(s) = \frac{1}{1-q^{1-s}},$$ and this provides a meromorphic continuation of $\zeta_q$ with a simple pole at $s=1$. As before, we will make the change of variables $u=q^{-s}$ and so the zeta-function becomes
$$ \mathcal{Z}(u) = \zeta_q(s) = \sum_{f \in\mathcal{M}} u^{\deg(f)} = \frac{1}{1-qu},$$ with a simple pole at $u=1/q$. Note that $\mathcal{Z}(u)$ can also be written as an Euler product
$$ \mathcal{Z}(u) = \prod_P \Big(1-u^{\deg(P)}\Big)^{-1},$$ where the product is over monic, irreducible polynomials in $\mathbb{F}_q[t]$.

The quadratic character over $\mathbb{F}_q[t]$ is defined as follows. For $P$ a monic, irreducible polynomial let
$$ \Big( \frac{f}{P} \Big)= 
\begin{cases}
1 & \mbox{ if } P \nmid f, f \text{ is a square modulo }f, \\
-1 & \mbox{ if } P \nmid f, f \text{ is not a square modulo }f, \\
0 & \mbox{ if } P|f.
\end{cases}
$$
We extend the definition of the quadratic residue symbol above to any $D \in \mathbb{F}_q[t]$ by multiplicativity, and define the quadratic character $\chi_D$ by
$$\chi_D(f) = \Big( \frac{D}{f} \Big).$$
Since we assumed that $q \equiv 1 \pmod 4$, note that the quadratic reciprocity holds; namely if $A$ and $B$ are two monic coprime polynomials, then 
$$ \Big( \frac{A}{B} \Big) = \Big( \frac{B}{A} \Big).$$

Throughout the paper, we will often make use of the Perron formula over function fields. If the series $\sum_{f \in \mathcal{M}} a(f) u^{\deg(f)}$ is absolutely convergent for $|u| \leq r <1$ then
$$ \sum_{f \in \mathcal{M}_n} a(f) = \frac{1}{2 \pi i} \oint_{|u|=r}  \bigg(\sum_{f \in \mathcal{M}} a(f) u^{\deg(f)}\bigg)\frac{du}{u^{n+1}},$$
and
$$ \sum_{f \in \mathcal{M}_{\leq n}} a(f) = \frac{1}{2 \pi i} \oint_{|u|=r} \bigg(\sum_{f \in \mathcal{M}} a(f) u^{\deg(f)}\bigg)\frac{du}{u^{n+1}(1-u)}.$$

Recall that the twisted elliptic curve $L$-function $\mathcal{L}(E \otimes \chi_D,u)$ is a polynomial of degree $\big(\mathfrak{n}+ 2 \deg(D)\big)$, with $\mathfrak{n}$ being defined in \eqref{degree}. Thus we can write
\begin{equation*}
\mathcal{L}(E \otimes \chi_D,u)=\sum_{n=0}^{\mathfrak{n}+ 2 \deg(D)}c_{n}u^{n},
\end{equation*}
where $c_{n}=\sum_{f\in M_{n}} \lambda(f)\chi_{D}(f)$. 

\begin{lemma}\label{coefficients relation}
The coefficients $c_{n}$ of $\mathcal{L}(E \otimes \chi_D,u)$ satisfy the following relation
$$
c_{n}=\epsilon\,q^{n-\mathfrak{n}/2- \deg(D)}c_{\mathfrak{n}+ 2 \deg(D)-n},
$$ with $\epsilon$ as in equation \eqref{tfe}.
In particular, if $\epsilon=-1$ and $\mathfrak{n}$ is even, then $c_{\mathfrak{n}/2+\deg(D)}=0$.
\end{lemma}
\begin{proof}
From the functional equation \eqref{tfe} we have
$$\sum_{n=0}^{\mathfrak{n}+ 2 \deg(D)}c_{n}u^{n}=\epsilon\,\sum_{n=0}^{\mathfrak{n}+ 2 \deg(D)}c_{n}q^{\mathfrak{n}/2+ \deg(D)-n}u^{\mathfrak{n}+ 2 \deg(D)-n}.$$
By setting $k:=\mathfrak{n}+ 2 \deg(D)-n$ we get
$$\sum_{n=0}^{\mathfrak{n}+ 2 \deg(D)}c_{n}u^{n}=\epsilon\,\sum_{k=0}^{\mathfrak{n}+ 2 \deg(D)}c_{\mathfrak{n}+ 2 \deg(D)-k}\,q^{k-\mathfrak{n}/2- \deg(D)}u^{k}.$$
Comparing the coefficients we obtain the lemma.
\end{proof}

For $D\in\mathcal{H}_{2g+1}^{*}$, we can obtain the following exact formulas for $L(E\otimes\chi_D,1/2)$ and $L'(E\otimes\chi_D,1/2)$. These are the analogues of the approximate functional equations in the number field setting.

\begin{lemma}\label{fe} Let $D \in \mathcal{H}_{2g+1}^{*}$. Then
$$ L ( E \otimes \chi_D, \tfrac{1}{2}) = \sum_{f \in \mathcal{M}_{\leq [\mathfrak{n}/2]+\deg(D)}} \frac{\lambda(f) \chi_D(f)}{\sqrt{|f|}}+\epsilon\, \sum_{f \in \mathcal{M}_{\leq [(\mathfrak{n}-1)/2]+\deg(D)}} \frac{\lambda(f) \chi_D(f)}{\sqrt{|f|}},$$ with $\epsilon$ as in \eqref{tfe}.
\end{lemma}
\begin{proof}
We use Lemma \ref{coefficients relation} to get
\begin{align*}\mathcal{L}(E \otimes \chi_D,u)&=\sum_{n=0}^{\mathfrak{n}+ 2 \deg(D)}c_{n}u^{n}=
\sum_{n=0}^{[\mathfrak{n}/2]+\deg(D)}c_{n}u^{n}+
\sum_{n=[\mathfrak{n}/2]+\deg(D)+1}^{\mathfrak{n}+ 2 \deg(D)}c_{n}u^{n}\\
&=\sum_{n=0}^{[\mathfrak{n}/2]+\deg(D)}c_{n}u^{n}+
\epsilon\,\sum_{n=[\mathfrak{n}/2]+\deg(D)+1}^{\mathfrak{n}+ 2 \deg(D)}c_{\mathfrak{n}+ 2 \deg(D)-n}\,q^{n-\mathfrak{n}/2- \deg(D)}u^{n}.
\end{align*}
Changing the summation variable in the second sum leads to
$$
\mathcal{L}(E \otimes \chi_D,u)=\sum_{n=0}^{[\mathfrak{n}/2]+\deg(D)}c_{n}u^{n}+
\epsilon\,\sum_{n=0}^{[(\mathfrak{n}-1)/2]+\deg(D)}c_{n}(qu^2)^{\mathfrak{n}/2+ \deg(D)-n}u^{n}.$$
Taking $u=q^{-1/2}$ and recalling that $c_{n}=\sum_{f\in M_{n}} \lambda(f)\chi_{D}(f)$ conclude the proof.
\end{proof}

\begin{lemma}\label{fe'} Let $D \in \mathcal{H}_{2g+1}^{*}$. If $\epsilon=-1$, then
$$ L' ( E \otimes \chi_D, \tfrac{1}{2}) =2(\log q)  \sum_{f \in \mathcal{M}_{\leq [\mathfrak{n}/2]+\deg(D)}} \frac{\big([\mathfrak{n}/2]+\deg(D)-\deg(f)\big)\lambda(f) \chi_D(f)}{\sqrt{|f|}}.$$
\end{lemma}
\begin{proof}
The above formula follows simply by differentiating the last equation in the proof of Lemma \ref{fe}. Just note that as remarked in Lemma \ref{coefficients relation}, if $\epsilon=-1$ and $\mathfrak{n}$ is even, then $c_{\mathfrak{n}/2+\deg(D)}=0$.
\end{proof}

We also have the following lemma which expresses a character sum over square-free polynomials in terms of sums over monics.

\begin{lemma}
We have
\begin{align*}
\sum_{D \in \mathcal{H}_{2g+1}^{*}} \chi_D(f) &= \sum_{C_1 | \Delta} \mu(C_1) \chi_{C_1}(f) \sum_{C_2 | (\Delta f)^{\infty}} \sum_{R \in \mathcal{M}_{2g+1- \deg(C_1) - 2 \deg(C_2)}} \chi_R(f) \\
&\qquad\qquad - q \sum_{C_1 | \Delta } \mu(C_1) \chi_{C_1}(f) \sum_{C_2 | (\Delta f)^{\infty}} \sum_{R \in \mathcal{M}_{2g-1- \deg(C_1) - 2 \deg(C_2)}} \chi_R(f) ,
\end{align*}
where by $C_2 | (\Delta f)^{\infty}$ we mean that the prime factors of $C_2$ divide $\Delta f$.
\label{sumd}
\end{lemma}

\begin{proof}
Let
$$ \mathcal{A}(u) = \sum_{\substack{D \in \mathcal{H} \\ (D,\Delta)=1}} \chi_D(f) u^{\deg(D)}.$$ 
Then
\begin{align*}
 \mathcal{A}(u)& = \displaystyle \prod_{P \nmid\Delta f} \Big(1+ \chi_P(f) u^{\deg(P)} \Big) =\prod_{P \nmid\Delta f} \Big(1- u^{2\deg(P)} \Big)\Big(1- \chi_P(f) u^{\deg(P)} \Big)^{-1}\\
& =(1-qu^2)\mathcal{L}(u,\chi_f)\prod_{P |\Delta f} \Big(1- u^{2\deg(P)} \Big)^{-1}\prod_{\substack{P | \Delta \\ P \nmid f}} \Big( 1- \chi_P(f) u^{\deg(P)}\Big).
\end{align*}
Writing
$$\prod_{P |\Delta f} \Big(1- u^{2\deg(P)} \Big)^{-1} = \sum_{C_2 | (\Delta f)^{\infty}} u^{2 \deg(C_2)},$$
$$\prod_{\substack{P | \Delta \\ P \nmid f}} \Big( 1-\chi_P(f) u^{\deg(P)} \Big)= \sum_{C_1 | \Delta} \mu(C_1) \chi_{C_1}(f) u^{\deg(C_1)},$$
and comparing the coefficients of $u^{2g+1}$, the conclusion follows.
\end{proof}
As in \cite{hayes} we define the exponential over function fields as follows. For $a \in \mathbb{F}_q\big((1/t)\big)$ let
$$ e_q(a) = \exp\bigg(\frac{ 2 \pi i  \text{Tr}_{F_{\mathbb{F}_q/\mathbb{F}_p}}(a_1)}{p}\bigg),$$ where $a_1$ is the coefficient of $1/t$ in the Laurent expansion of $a$ and $q$ is a power of the prime $p$. We define the generalized quadratic Gauss sum as
$$G(V,f) = \sum_{u(\text{mod}\ f)} \chi_f(u) e_q \Big( \frac{uV}{f} \Big),$$ where $\chi_f$ is the quadratic character defined before. We gather here a few useful facts about $G(V,f)$ whose proofs can be found in \cite{fl1}. 

\begin{lemma}\label{Gauss}
\begin{enumerate}
\item If $(f,h)=1$, then $G(V, fh)= G(V, f) G(V,h)$.
\item Write $V= V_1 P^{\alpha}$ where $P \nmid V_1$.
Then 
 $$G(V , P^j)= 
\begin{cases}
0 & \mbox{if } j \leq \alpha \text{ and } j \text{ odd,} \\
\phi(P^j) & \mbox{if }  j \leq \alpha \text{ and } j \text{ even,} \\
-|P|^{j-1} & \mbox{if }  j= \alpha+1 \text{ and } j \text{ even,} \\
\chi_P(V_1) |P|^{j-1/2} & \mbox{if } j = \alpha+1 \text{ and } j \text{ odd, } \\
0 & \mbox{if } j \geq 2+ \alpha .
\end{cases}$$ 
\end{enumerate}
\end{lemma}

The following Poisson summation formula in function fields holds.
\begin{lemma}
Let $f\in\mathcal{M}_n$. If $n$ is even, then
$$ \sum_{R \in \mathcal{M}_m} \chi_R(f) = \frac{q^m}{|f|} \bigg( G(0,f) + (q-1) \sum_{V \in \mathcal{M}_{\leq n-m-2}} G(V,f) - \sum_{V \in \mathcal{M}_{n-m-1}} G(V,f) \bigg),$$
otherwise
$$ \sum_{R \in \mathcal{M}_m} \chi_R(f) = \frac{q^m \overline{\tau(q)}}{|f|} \sum_{V \in \mathcal{M}_{n-m-1}} G(V,f),$$ where
$$\tau(q) = \sum_{c=1}^{q-1} \chi_f(c) \exp \Big(  \frac{2 \pi i \emph{Tr}_{\mathbb{F}_q/\mathbb{F}_p}( c)}{p} \Big)$$ is the usual Gauss sum over $\mathbb{F}_q$.
\label{poisson}
\end{lemma}
%\begin{lemma}
%For $f \in \mathbb{F}_q[x],$ we have
%$$ \sum_{D \in \hstar} \chi_D(f^2)= | \hstar | \prod_{\substack{P | f \\ P \nmid \Delta}} \frac{|P|}{|P|+1} + O(\tau(f)).$$
%\label{sumsquares}
%\end{lemma}

\subsection{Outline of the proof}
We will use the approximate functional equations for the $L$--functions involved in the moment computations and then truncate the Dirichlet series close to the endpoint. For the longer Dirichlet series, we will use Poisson summation and standard techniques to compute the main terms. For the tails, we will go back and write the Dirichlet series in terms of expressions involving moments and then use upper bounds for moments. The key in bounding the tails is the fact that the moments behave differently depending on the points where we evaluate them (the power of $g$ gets smaller in different ranges).
\section{Main proposition}\label{mainpropS}

For $N|\Delta^\infty$, let 
\begin{equation}\label{formulaS}
S_{E_1,E_2}(N,X,Y;\alpha,\beta):= \sum_{D \in \hstar} \sum_{\substack{f \in \mathcal{M}_{\leq X}\\h \in \mathcal{M}_{\leq Y}}} \frac{ \lambda_1(f) \lambda_2(h)\chi_D(Nfh)}{ |f|^{1/2+\alpha} |h|^{1/2 + \beta}}.
\end{equation}

\begin{proposition}\label{prop1}
Assume $X\geq Y$. We have
\begin{align*}
S_{E}(N,X,Y):&=S_{E,E}(N,X,Y;0,0)\\
&=|\mathcal{H}_{2g+1}|\,\mathcal{C}_E(N;1,1,1)\,L\big(\emph{Sym}^2 E, 1\big)^3\, Y\\
&\qquad\qquad+O\big(q^{2g}\big)+O\big(q^{2g-Y/5}g^2\big)+O\big(q^{g/2+3(X+Y)/8}g^{30}\big),
\end{align*}
and if $E_1\ne E_2$, then
\begin{align*}
S_{E_1,E_2}(N,X,Y;\alpha,\beta)&=\,|\mathcal{H}_{2g+1}|\,\mathcal{C}_{E_1,E_2}(N;1,1,1,\alpha,\beta)\,L\big(\emph{Sym}^2 E_1, 1+2\alpha\big)L\big( \emph{Sym}^2 E_2,1+2\beta\big)\\
&\qquad\qquad L\big ( E_1 \otimes E_2 , 1+\alpha+\beta\big)+O\big(q^{2g-Y/5}g^2\big)+O\big(q^{g/2+3(X+Y)/8}g^{30}\big)
\end{align*}
uniformly for $|\alpha|,|\beta|\leq 1/g$, where the values $\mathcal{C}_E(N;1,1,1)$ and $\mathcal{C}_{E_1,E_2}(N;1,1,1,\alpha,\beta)$ are defined in \eqref{Euler1}, \eqref{Euler2} and \eqref{Euler3}.
\end{proposition}

We begin the proof of the proposition by applying Lemma \ref{sumd} and rewriting $S_{E_1,E_2}(N,X,Y;\alpha,\beta)$ as
\begin{align}\label{formulaS1}
&  \sum_{\substack{f \in \mathcal{M}_{\leq X}\\h \in \mathcal{M}_{\leq Y}}} \frac{ \lambda_1(f) \lambda_2(h)}{ |f|^{1/2+\alpha} |h|^{1/2 + \beta}}  \sum_{C_1 | \Delta} \mu(C_1) \chi_{C_1}(Nfh)\sum_{C_2 | (\Delta fh)^{\infty}} \sum_{R \in \mathcal{M}_{2g+1- \deg(C_1) - 2 \deg(C_2)}} \chi_R(Nfh) \nonumber \\
&\ - q  \sum_{\substack{f \in \mathcal{M}_{\leq X}\\h \in \mathcal{M}_{\leq Y}}} \frac{ \lambda_1(f) \lambda_2(h)}{ |f|^{1/2+\alpha} |h|^{1/2 + \beta}}  \sum_{C_1 | \Delta} \mu(C_1) \chi_{C_1}(Nfh) \sum_{C_2 | (\Delta fh)^{\infty}} \sum_{R \in \mathcal{M}_{2g-1- \deg(C_1) - 2 \deg(C_2)}} \chi_R(Nfh)\nonumber\\
&=S_{E_1,E_2}(N,X,Y,Z;\alpha,\beta) + T_{E_1,E_2}(N,X,Y,Z;\alpha,\beta),
\end{align}
where $S_{E_1,E_2}(N,X,Y,Z;\alpha,\beta)$ denotes the contribution of the terms with $\deg(C_2)\leq Z$ and $T_{E_1,E_2}(N,X,Y,Z;\alpha,\beta)$ denotes that with $\deg(C_2)> Z$ for some $Z\leq g$. We first estimate $T_{E_1,E_2}(N,X,Y,Z;\alpha,\beta)$, which is easier.

\subsection{The term $T_{E_1,E_2}(N,X,Y,Z;\alpha,\beta)$}

\begin{lemma}\label{T'}
We have
\[
T_{E_1,E_2}(N,X,Y,Z;\alpha,\beta) \ll_\varepsilon q^{2g-3Z/2}g^{6+\varepsilon}
\]
uniformly for $|\alpha|,|\beta|\leq 1/g$.
\end{lemma}
\begin{proof}
It suffices to prove the bound for $T_{E_1,E_2}^{'}$, which is
\[
 \sum_{\substack{f \in \mathcal{M}_{\leq X}\\h \in \mathcal{M}_{\leq Y}}} \frac{ \lambda_1(f) \lambda_2(h)}{ |f|^{1/2+\alpha} |h|^{1/2 + \beta}}  \sum_{C_1 | \Delta} \mu(C_1) \chi_{C_1}(Nfh)\sum_{\substack{C_2 | (\Delta fh)^{\infty}\\\deg(C_2)>Z}} \sum_{R \in \mathcal{M}_{2g+1- \deg(C_1) - 2 \deg(C_2)}} \chi_R(Nfh).
\]
We use the Perron formula for the sum over $f$ and $h$. We write $\text{rad}(C_2)/(\text{rad}(C_2) ,\Delta)=C=AB$, and replace $f,h$ by $Af$ and $Bh$, respectively. Then
\begin{align*}
T_{E_1,E_2}^{'}&=\sum_{\substack{c_2> Z\\c_1+2c_2\leq 2g+1}}\sum_{\substack{C_1\in\mathcal{M}_{c_1}\\C_2\in\mathcal{M}_{c_2}\\C_1 | \Delta,\,C=AB}}  \frac{ \mu(C_1) \chi_{C_1}(NAB)}{|A|^{1/2+\alpha} |B|^{1/2 + \beta}}\sum_{R \in \mathcal{M}_{2g+1- c_1- 2c_2}} \chi_R(NAB)\nonumber\\
&\qquad\times \frac{1}{(2\pi i)^2}\oint_{|u|=r}\oint_{|v|=r} \sum_{f,h\in\mathcal{M}}\frac{ \chi_{C_1R}(fh)\lambda_1(Af) \lambda_2(Bh)u^{\deg(f)}v^{\deg(h)}}{ |f|^{1/2+\alpha} |h|^{1/2 + \beta}} \, \\
&\qquad\qquad\qquad\times \frac{dudv}{u^{X-\deg(A)+1}v^{Y-\deg(B)+1}(1-u)(1-v)}
\end{align*}
for any $r<1$. The sum over $f$ and $h$ may be written as
\[
\mathcal{D}_1(A,B,C_1R;u,v,\alpha,\beta)\mathcal{L}\Big(E_1\otimes\chi_{C_1R},\frac{u}{q^{1/2+\alpha}}\Big)\mathcal{L}\Big(E_2\otimes\chi_{C_1R},\frac{v}{q^{1/2+\beta}}\Big),
\]
where $\mathcal{D}_1(A,B,C_1R;u,v,\alpha,\beta)$ is some Euler product which is uniformly convergent provided that $|u|,|v|\leq q^{-1/g}$, and satisfies
\begin{align*}
&\mathcal{D}_1(A,B,C_1R;u,v,\alpha,\beta)\ll \tau(AB)
\end{align*}
uniformly in this region. Moving the $u$ and $v$ contours to $|u|=|v|= q^{-1/g}$ and using the bound
\begin{align*}
& \Big| \mathcal{L}\Big(E_1\otimes\chi_{C_1R},\frac{u}{q^{1/2+\alpha}}\Big)\mathcal{L}\Big(E_2\otimes\chi_{C_1R},\frac{v}{q^{1/2+\beta}}\Big) \Big| \\
&\qquad\qquad\ll \Big|\mathcal{L}\Big(E_1\otimes\chi_{C_1R},\frac{u}{q^{1/2+\alpha}}\Big)\Big|^2+\Big|\mathcal{L}\Big(E_2\otimes\chi_{C_1R},\frac{v}{q^{1/2+\beta}}\Big)\Big|^2
\end{align*}
we get
\begin{align*}
T_{E_1,E_2}^{'}&\ll g^2\sum_{\substack{c_2> Z\\c_1+2c_2\leq 2g+1}}\sum_{\substack{C_1\in\mathcal{M}_{c_1}\\C_2\in\mathcal{M}_{c_2}\\C_1 | \Delta}}  \frac{\tau(C)^2}{\sqrt{|C|}}\oint_{|u|=q^{-1/g}}\oint_{|v|=q^{-1/g}}\nonumber\\
& \qquad  \qquad\sum_{R \in \mathcal{M}_{2g+1- c_1- 2c_2}} \bigg(\, \Big|\mathcal{L}\Big(E_1\otimes\chi_{C_1R},\frac{u}{q^{1/2+\alpha}}\Big)\Big|^2  +  \Big|\mathcal{L}\Big(E_2\otimes\chi_{C_1R},\frac{v}{q^{1/2+\beta}}\Big)\Big|^2 \bigg) \,  dudv.
\end{align*}

Now let $D=(C_1,R)$. Write $R=DR_1=DEH^2$, where $E$ is square-free, and let $C_1=DC_3$. Then
$$ \Big|  \mathcal{L} \Big(E_1 \otimes \chi_{C_1R} , \frac{u}{q^{1/2+\alpha}}\Big) \Big| \ll_\varepsilon  |DH|^{\varepsilon}\Big|  \mathcal{L} \Big(E_1 \otimes \chi_{C_3E} , \frac{u}{q^{1/2+\alpha}}\Big) \Big|.$$
Using upper bounds for moments (see Remark \ref{rem1} after Theorem \ref{thm-ub}), we get that

\[
T_{E_1,E_2}^{'}\ll_\varepsilon q^{2g}g^{3+\varepsilon}\sum_{\substack{C_2\in\mathcal{M}\\Z<\deg(C_2)\leq g}}\frac{\tau(C_2)^2}{\sqrt{|\text{rad}(C_2)|}\,C_2^2}  \ll_\varepsilon q^{2g-3Z/2}g^{6+\varepsilon},
\]
and this finishes the proof of Lemma \ref{T'}. 

\end{proof}

\subsection{The term $S_{E_1,E_2}(N,X,Y,Z;\alpha,\beta)$}

Define $S_{E_1,E_2}^{'}(N,X,Y,Z;\alpha,\beta)$ to be
\[
 \sum_{\substack{f \in \mathcal{M}_{\leq X}\\h \in \mathcal{M}_{\leq Y}}} \frac{ \lambda_1(f) \lambda_2(h)}{ |f|^{1/2+\alpha} |h|^{1/2 + \beta}}  \sum_{C_1 | \Delta} \mu(C_1) \chi_{C_1}(Nfh)\sum_{\substack{C_2 | (\Delta fh)^{\infty}\\\deg(C_2)\leq Z}} \sum_{R \in \mathcal{M}_{2g+1- \deg(C_1) - 2 \deg(C_2)}} \chi_R(Nfh) ,
\]
and $S_{E_1,E_2}^{''}(N,X,Y,Z;\alpha,\beta)$ to be the same sum with $g$ being replaced by $(g-1)$. Then
\[
S_{E_1,E_2}(N,X,Y,Z;\alpha,\beta)=S_{E_1,E_2}^{'}(N,X,Y,Z;\alpha,\beta)-qS_{E_1,E_2}^{''}(N,X,Y,Z;\alpha,\beta).
\]

Using Lemma \ref{poisson} on the sum over $R$, it follows that $S_{E_1,E_2}^{'}(N,X,Y,Z;\alpha,\beta)$ equals
\begin{align*}
q^{2g+1} &   \sum_{\substack{f \in \mathcal{M}_{\leq X}\\h \in \mathcal{M}_{\leq Y}\\ \deg(Nfh)\ \text{even}}} \frac{ \lambda_1(f) \lambda_2(h)}{ |N||f|^{3/2+\alpha} |h|^{3/2 + \beta}}  \sum_{\substack{C_1 | \Delta\\C_2 | (\Delta fh)^{\infty} \\ \deg(C_2)\leq Z\\ \deg(C_1)+ 2 \deg(C_2) \leq 2g+1}}  \frac{ \mu(C_1) \chi_{C_1}(Nfh)}{|C_1||C_2|^2}\nonumber  \\
&\times\bigg( G(0,Nfh) + (q-1) \sum_{V \in \mathcal{M}_{\leq \deg(Nfh) + \deg(C_1) + 2 \deg(C_2) - 2g-3}} G(V,Nfh) \nonumber\\
&-\sum_{V \in \mathcal{M}_{ \deg(Nfh) + \deg(C_1) + 2 \deg(C_2) - 2g-2}}  G(V,Nfh) \bigg)\\
&+q^{2g+1} \overline{\tau(q)}  \sum_{\substack{f \in \mathcal{M}_{\leq X}\\h \in \mathcal{M}_{\leq Y}\\ \deg(Nfh)\ \text{odd}}} \frac{ \lambda_1(f) \lambda_2(h)}{ |N||f|^{3/2+\alpha} |h|^{3/2 + \beta}}  \sum_{\substack{C_1 | \Delta\\C_2 | (\Delta fh)^{\infty} \\  \deg(C_2)\leq Z\\\deg(C_1)+ 2 \deg(C_2) \leq 2g+1}}  \frac{ \mu(C_1) \chi_{C_1}(Nfh)}{|C_1||C_2|^2}\nonumber\\
&\times\sum_{V \in \mathcal{M}_{ \deg(Nfh) + \deg(C_1) + 2 \deg(C_2) - 2g-2}}  G(V,Nfh). \nonumber
\end{align*}

Let $S_{E_1,E_2}^{'}(V=0)$ denote the terms with $V=0$ above and $S_{E_1,E_2}^{'}(V \neq 0)$ be the terms with non-zero $V$. The terms $S_{E_1,E_2}^{''}(V=0)$ and $S_{E_1,E_2}^{''}(V \neq 0)$ are similarly defined. Let $$S_{E_1,E_2}\big(N,X,Y,Z;\alpha,\beta;V=0\big) = S_{E_1,E_2}^{'}(V=0)-q\,S_{E_1,E_2}^{''}(V=0)$$
and $$S_{E_1,E_2}\big(N,X,Y,Z;\alpha,\beta;V\ne 0\big) = S_{E_1,E_2}^{'}(V\ne0)-q\,S_{E_1,E_2}^{''}(V\ne0)$$ so that we have
\begin{equation}\label{formulaS2}
S_{E_1,E_2}(N,X,Y,Z;\alpha,\beta)=S_{E_1,E_2}\big(N,X,Y,Z;\alpha,\beta;V=0\big)+S_{E_1,E_2}\big(N,X,Y,Z;\alpha,\beta;V\ne 0\big).
\end{equation}
We shall evaluate $S_{E_1,E_2}\big(N,X,Y,Z;\alpha,\beta;V=0\big)$ in Section \ref{V=0} and bound $S_{E_1,E_2}\big(N,X,Y,Z;\alpha,\beta;V\ne 0\big)$ in Section \ref{V<>0}.

\subsection{The $V=0$ terms}\label{V=0}

\begin{lemma}
We have
\begin{align*}
S_{E}\big(N,X,Y,Z;V=0\big):&=S_{E,E}\big(N,X,Y,Z;0,0;V=0\big)\\
&=|\mathcal{H}_{2g+1}|\,\mathcal{C}_E(N;1,1,1)\,L\big(\emph{Sym}^2 E, 1\big)^3\, Y+O\big(q^{2g}\big)+O\big(q^{2g-Y/5}g^2\big),
\end{align*}
and if $E_1\ne E_2$, then
\begin{align*}
&S_{E_1,E_2}\big(N,X,Y,Z;\alpha,\beta;V=0\big)=\,|\mathcal{H}_{2g+1}|\,\mathcal{C}_{E_1,E_2}(N;1,1,1,\alpha,\beta)\,L\big(\emph{Sym}^2 E_1, 1+2\alpha\big)\\
&\qquad\qquad L\big( \emph{Sym}^2 E_2,1+2\beta\big)L\big ( E_1 \otimes E_2 , 1+\alpha+\beta\big)+O\big(q^{2g-Y/5}g^2\big)+O\big(q^{2g-3Z/2}\big)
\end{align*}
uniformly for $|\alpha|,|\beta|\leq 1/g$, where the values $\mathcal{C}_E(N;1,1,1)$ and $\mathcal{C}_{E_1,E_2}(N;1,1,1,\alpha,\beta)$ are defined in \eqref{Euler1}, \eqref{Euler2} and \eqref{Euler3}.
\end{lemma}
\begin{proof}
Note that $G(0,Nfh) \neq 0$ if and only if $Nfh$ is a square polynomial, and in this case $G(0,Nfh) = \phi(Nfh)$. Hence
\begin{align}
S_{E_1,E_2}^{'}(V=0) = q^{2g+1} \sum_{\substack{f \in \mathcal{M}_{\leq X}\\h \in \mathcal{M}_{\leq Y}\\ Nfh= \square}} \frac{ \lambda_1(f) \lambda_2(h) \phi(Nfh)}{|N| |f|^{3/2+\alpha} |h|^{3/2 + \beta}}\sum_{\substack{C_2 | (\Delta fh)^{\infty} \\ \deg(C_2) \leq Z}} \frac{1}{|C_2|^2}\sum_{\substack{C_1 | \Delta \\ (C_1,Nfh)=1 \\ \deg(C_1) \leq 2g+1 - 2 \deg(C_2)}}  \frac{ \mu(C_1) }{|C_1|}.
\label{main}
\end{align}
We have
\begin{align}\label{error1}
&\sum_{\substack{ C_1 | \Delta \\ (C_1, Nfh)=1 \\ \deg(C_1) \leq 2g+1 - 2 \deg(C_2)}}  \frac{ \mu(C_1)}{|C_1|} = \sum_{\substack{ C_1 | \Delta \\ (C_1, Nfh)=1  }} \frac{ \mu(C_1)}{|C_1|} - \sum_{\substack{ C_1 | \Delta \\ (C_1, Nfh)=1 \\ \deg(C_1) > 2g+1 - 2 \deg(C_2)}} \frac{ \mu(C_1)}{|C_1|}\nonumber\\
&\quad =  \prod_{\substack{P | \Delta \\ P \nmid Nfh}} \bigg(1 - \frac{1}{|P|} \bigg) +  O \big(q^{-2g}|C_2|^2 \big)=\frac{|Nfh|}{\phi(Nfh)}  \prod_{P | \Delta fh} \bigg(1 - \frac{1}{|P|} \bigg) +  O \big(q^{-2g}|C_2|^2 \big).   
\end{align}
Note that
\begin{equation}\label{C2formula}
 \sum_{\substack{C_2 \in \mathcal{M}_n \\ C_2 | (\Delta fh)^{\infty}}} 1 \ll_\varepsilon q^{ \varepsilon n}.
 \end{equation}
Let $N=N_1^2N_2$ with $N_2$ being square-free. The condition $Nfh=\square$ is equivalent to $fh=N_2\ell^2$ for some polynomial $\ell$. Then we can write $f = N_2' A$ and $h=N_2''B$, with $N_2'N_2''=N_2$ and $AB=\ell^2$. 
It follows that the contribution of the error term in \eqref{error1} to \eqref{main} will be
\begin{align*}
\ll_\varepsilon q^{\varepsilon g} \sum_{\ell \in \mathcal{M}_{\leq (X+Y)/2}} \frac{1}{|\ell|} \sum_{AB =\ell^2} \big| \lambda_1(N_2'A) \lambda_2(N_2''B) \big|  \ll_\varepsilon q^{\varepsilon g},
\end{align*}by using the bound $ | \lambda_i(f) | \leq \tau(f)\ll_\varepsilon| f|^\varepsilon.$
Thus we can rewrite \eqref{main} as
\begin{equation*}
S_{E_1,E_2}^{'}(V=0) = q^{2g+1}  \sum_{\substack{f \in \mathcal{M}_{\leq X}\\h \in \mathcal{M}_{\leq Y} \\Nfh= \square}} \frac{ \lambda_1(f) \lambda_2(h)}{|f|^{1/2+\alpha} |h|^{1/2 + \beta}} \prod_{P | \Delta fh} \bigg(1- \frac{1}{|P|} \bigg)\sum_{\substack{C_2 | (\Delta fh)^{\infty} \\  \deg(C_2) \leq Z}} \frac{1}{|C_2|^2} + O_\varepsilon(q^{\varepsilon g}).
\end{equation*}

We obtain a similar estimate for $S_{E_1,E_2}^{''}(V=0)$ with $g$ being replaced by $(g-1)$, and hence
\begin{align}\label{formulaS3}
&S_{E_1,E_2}\big(N,X,Y,Z;\alpha,\beta;V=0\big)\nonumber\\
&\qquad=|\mathcal{H}_{2g+1}|  \sum_{\substack{f \in \mathcal{M}_{\leq X}\\h \in \mathcal{M}_{\leq Y} \\ Nfh= \square}} \frac{ \lambda_1(f) \lambda_2(h)}{|f|^{1/2+\alpha} |h|^{1/2 + \beta}} \prod_{P | \Delta fh} \bigg(1- \frac{1}{|P|} \bigg)\sum_{\substack{C_2 | (\Delta fh)^{\infty} \\  \deg(C_2) \leq Z}} \frac{1}{|C_2|^2} + O_\varepsilon(q^{\varepsilon g}).
\end{align}
From the Perron formula for the sum over $C_2$,
\[
\sum_{\substack{C_2 | (\Delta fh)^{\infty} \\  \deg(C_2) \leq Z}} \frac{1}{|C_2|^2}=\frac{1}{2\pi i}\oint_{|w|=r}\prod_{P|\Delta fh}\bigg(1-\frac{w^{\deg(P)}}{|P|^2}\bigg)^{-1}\frac{dw}{w^{Z+1}(1-w)}
\]
for any $r<1$,

 it follows that
\begin{align*}
&S_{E_1,E_2}\big(N,X,Y,Z;\alpha,\beta;V=0\big)=\frac{|\mathcal{H}_{2g+1}| }{(2\pi i)^3}\oint_{|u|=r}\oint_{|v|=r}\oint_{|w|=r} \\
&\qquad\qquad\qquad\qquad\mathcal{A}_{E_1,E_2}(N;u,v,w,\alpha,\beta)\frac{dudvdw}{u^{X+1}v^{Y+1}w^{Z+1}(1-u)(1-v)(1-w)} + O_\varepsilon(q^{\varepsilon g}),
\end{align*}
where
$$ \mathcal{A}_{E_1,E_2}(N;u,v,w,\alpha,\beta) = \sum_{\substack{f,h \in \mathcal{M} \\ Nfh = \square}} \frac{ \lambda_1(f)\lambda_2(h)u^{\deg(f)} v^{\deg(h)}}{|f|^{1/2+\alpha} |h|^{1/2 + \beta}}\prod_{P | \Delta fh} \bigg(1- \frac{1}{|P|} \bigg)\bigg(1-\frac{w^{\deg(P)}}{|P|^2}\bigg)^{-1}.$$ 

We can write down an Euler product for $\mathcal{A}_{E_1,E_2}(N;u,v,w,\alpha,\beta)$ as follows.
\begin{align}\label{Euler1}
&\mathcal{A}_{E_1,E_2}(N;u,v,w,\alpha,\beta) \nonumber\\
&\quad= \prod_{P \nmid \Delta} \bigg( 1+\bigg(1- \frac{1}{|P|} \bigg)\bigg(1-\frac{w^{\deg(P)}}{|P|^2}\bigg)^{-1} \sum_{i+j\ \text{even}\, \geq\,2}\frac{ \lambda_1(P^i) \lambda_2(P^{j}) u^{i \deg(P)} v^{j\deg(P)}}{|P|^{(1/2+\alpha)i+(1/2+\beta)j}}\bigg)\\
&\qquad\times \prod_{P|\Delta} \bigg( \bigg(1- \frac{1}{|P|} \bigg)\bigg(1-\frac{w^{\deg(P)}}{|P|^2}\bigg)^{-1}\sum_{\substack{i,j\\i+j+\text{ord}_P(N)\ \text{even}}}\frac{ \lambda_1(P^i) \lambda_2(P^{j}) u^{i \deg(P)} v^{j\deg(P)}}{|P|^{(1/2+\alpha)i+(1/2+\beta)j}} \bigg).\nonumber
\end{align}
Then
\begin{align}\label{Euler2}
\mathcal{A}_{E}(N;u,v,w,0,0) =\mathcal{C}_E(N;u,v,w)\,\mathcal{L} \Big(\text{Sym}^2 E, \frac{u^2}{q}\Big)\mathcal{L}\Big( \text{Sym}^2 E,\frac{v^2}{q}\Big)\mathcal{L}\Big (\text{Sym}^2 E , \frac{uv}{q}\Big)\mathcal{Z}\Big (\frac{uv}{q}\Big),
\end{align}
and
\begin{align}\label{Euler3}
&\mathcal{A}_{E_1,E_2}(N;u,v,w,\alpha,\beta) \\
&\qquad=\mathcal{C}_{E_1,E_2}(N;u,v,w,\alpha,\beta) \mathcal{L} \Big(\text{Sym}^2 E_1, \frac{u^2}{q^{1+2\alpha}}\Big) \, \mathcal{L}\Big( \text{Sym}^2 E_2,\frac{v^2}{q^{1+2\beta}}\Big)\mathcal{L}\Big ( E_1 \otimes E_2 , \frac{uv}{q^{1+\alpha+\beta}}\Big),\nonumber
\end{align} if $E_1\ne E_2$, where $\mathcal{C}_E(N;u,v,w)$ and $\mathcal{C}_{E_1,E_2}(N;u,v,w,\alpha,\beta)$ are some Euler products which are uniformly bounded for example when $|u|,|v|\leq q^{1/5}$, $|w|\leq q^{3/2}$.

Consider the case $E_1=E_2=E$. We have
\begin{align*}
&S_{E}\big(N,X,Y,Z;V=0\big)=\,\frac{|\mathcal{H}_{2g+1}| }{(2\pi i)^3}\oint_{|u|=r}\oint_{|v|=r}\oint_{|w|=r}\,\mathcal{C}_E(N;u,v,w)\,\mathcal{L} \Big(\text{Sym}^2 E, \frac{u^2}{q}\Big)\\
&\qquad  \mathcal{L}\Big( \text{Sym}^2 E,\frac{v^2}{q}\Big)\mathcal{L}\Big (\text{Sym}^2 E, \frac{uv}{q}\Big)\frac{dudvdw}{u^{X+1}v^{Y+1}w^{Z+1}(1-u)(1-v)(1-w)\big(1-uv\big)}+ O_\varepsilon(q^{\varepsilon g})
\end{align*}
for any $r<1$. We choose $r=q^{-1/g}$ and move the $u$ contour to $|u|=q^{1/5}$, encountering two simple poles at $u=1$ and $u=1/v$. The new integral is trivially bounded by $O\big(q^{2g-X/5}g^2\big)$. 

Furthermore, the contribution from the residue at $u=1/v$ is
\begin{align*}
&-|\mathcal{H}_{2g+1}|\,L\big (\text{Sym}^2 E,1\big)\frac{1}{(2\pi i)^2}\oint_{|v|=q^{-1/g}}\oint_{|w|=q^{-1/g}}\mathcal{C}_E(N;1/v,v,w)\\
&\qquad\times \mathcal{L} \Big(\text{Sym}^2 E, \frac{1}{v^2q}\Big)\mathcal{L}\Big( \text{Sym}^2 E,\frac{v^2}{q}\Big)\frac{dvdw}{v^{Y-X}w^{Z+1}(1-v)^2(1-w)},
\end{align*}
which is $O(q^{2g})$. This can be seen by first moving the $v$ contour to $|v|=q^{-1/5}$, creating no poles, and then moving the $w$ contour to $|w|=q^{3/2}$, crossing a simple pole at $w=1$. Both the new integral and the residue at $w=1$ are $O(q^{2g})$  as $X\geq Y$. So
\begin{align*}
&S_{E}\big(N,X,Y,Z;V=0\big)=|\mathcal{H}_{2g+1}|\,L\big (\text{Sym}^2 E,1\big)\frac{1}{(2\pi i)^2}\oint_{|v|=q^{-1/g}}\oint_{|w|=q^{-1/g}}\,\mathcal{C}_E(N;1,v,w)\,\\
&\qquad\times \mathcal{L}\Big( \text{Sym}^2 E,\frac{v^2}{q}\Big)\mathcal{L}\Big (\text{Sym}^2 E, \frac{v}{q}\Big)\frac{dvdw}{v^{Y+1}w^{Z+1}(1-v)^2(1-w)}+O\big(q^{2g}\big)+O\big(q^{2g-X/5}g^2\big).
\end{align*}

We now move the $v$ contour to  $|v|=q^{1/5}$, encountering a double pole at $v=1$. The new integral is bounded by $O\big(q^{2g-Y/5}g\big)$, and an argument similar to the above implies that the residue at $v=1$ is
\[
|\mathcal{H}_{2g+1}|L\big(\text{Sym}^2 E,1\big)^3\, Y \frac{1}{2\pi i}\oint_{|w|=q^{-1/g}}\,\mathcal{C}_E(N;1,1,w)\frac{dw}{w^{Z+1}(1-w)}+O\big(q^{2g}\big).
\]

Hence
\begin{align*}
S_{E}\big(N,X,Y,Z;V=0\big)=&\,|\mathcal{H}_{2g+1}|\,\mathcal{C}_E(N;1,1,1) L\big(\text{Sym}^2 E,1\big)^3\, Y+O\big(q^{2g}\big)+O\big(q^{2g-Y/5}g^2\big).
\end{align*}

For $E_1\ne E_2$, we have that  
\begin{align*}
& S_{E_1,E_2}\big(N,X,Y,Z;\alpha,\beta;V=0\big)\\
&\qquad =\frac{|\mathcal{H}_{2g+1}|}{(2\pi i)^3}\oint_{|u|=r}\oint_{|v|=r}\oint_{|w|=r}\,\mathcal{C}_{E_1,E_2}(N;u,v,w,\alpha,\beta)\mathcal{L} \Big(\text{Sym}^2 E_1, \frac{u^2}{q^{1+2\alpha}}\Big)\mathcal{L}\Big( \text{Sym}^2 E_2,\frac{v^2}{q^{1+2\beta}}\Big)\\
& \qquad\qquad\times  \mathcal{L}\Big ( E_1 \otimes E_2 , \frac{uv}{q^{1+\alpha+\beta}}\Big)\frac{dudvdw}{u^{X+1}v^{Y+1}w^{Z+1}(1-u)(1-v)(1-w)}+O_\varepsilon \big(q^{\varepsilon g} \big)
\end{align*}
for any $r<1$. We choose $r=q^{-1/g}$ and first shift the $u$ contour to $|u|=q^{1/5}$, encountering a pole at $u=1$. The new  integral over $|u|=q^{1/5}$, $|v|=|w|=q^{-1/g}$ is bounded by $q^{2g-X/5}g^2$. To calculate the residue at $u=1$, we move the $v$ contour to $|v|=q^{1/5}$, crossing a pole at $v=1$. The new integral is $O(q^{2g-Y/5}g)$. For the residue at $u=v=1$, we move the $w$ contour to $|w|=q^{3/2}$. In doing so we obtain
\begin{align*}
&S_{E_1,E_2}\big(N,X,Y,Z;\alpha,\beta;V=0\big)=\,|\mathcal{H}_{2g+1}|\,\mathcal{C}_{E_1,E_2}(N;1,1,1,\alpha,\beta)\,L \big(\text{Sym}^2 E_1, 1+2\alpha\big)\\
&\qquad\times L\big( \text{Sym}^2 E_2,1+2\beta\big)L\big ( E_1 \otimes E_2 ,1+\alpha+\beta\big)+O\big(q^{2g-Y/5}g^2\big)+O\big(q^{2g-3Z/2}\big),
\end{align*}
and this concludes the proof of the lemma.
\end{proof}

\subsection{The $V\ne0$ terms}\label{V<>0}\label{propV<>0}

\begin{lemma}\label{lemmaV<>0}
We have
\begin{align*}
S_{E_1,E_2}\big(N,X,Y,Z;\alpha,\beta;V\ne0\big)\ll q^{(X+Y+Z)/2}g^{30}
\end{align*}
uniformly for $|\alpha|,|\beta|\leq 1/g$.
\end{lemma}
\begin{proof}
We will prove the bound for the term
\begin{align}\label{formulaS4}
S(V\ne0) =&\,q^{2g+1} \overline{\tau(q)}  \sum_{\substack{f \in \mathcal{M}_{\leq X}\\h \in \mathcal{M}_{\leq Y}\\ \deg(Nfh)\ \text{odd}}} \frac{ \lambda_1(f) \lambda_2(h)}{ |N||f|^{3/2+\alpha} |h|^{3/2 + \beta}}  \sum_{\substack{C_1 | \Delta\\C_2 | (\Delta fh)^{\infty} \\ \deg(C_2)\leq Z\\ \deg(C_1)+ 2 \deg(C_2) \leq 2g+1}}  \frac{ \mu(C_1) \chi_{C_1}(Nfh)}{|C_1||C_2|^2}\nonumber\\
&\qquad\times \sum_{V \in \mathcal{M}_{ \deg(Nfh) + \deg(C_1) + 2 \deg(C_2) - 2g-2}}  G(V,Nfh), 
\end{align}
the treatment of the other terms being similar. We also assume for simplicity that $\deg(N)$, $X$ and $Y$ are all odd. The other cases can be done similarly.

We use the Perron formula in the forms
$$ \sum_{\substack{f \in \mathcal{M}_{\leq X}\\\deg(f)\ \text{odd}} } a(f) = \frac{1}{2 \pi i} \oint_{|u|=r} \bigg(\sum_{f \in \mathcal{M}} a(f) u^{\deg(f)}\bigg)\frac{du}{u^{X+1}(1-u^2)}$$
and
$$ \sum_{\substack{f \in \mathcal{M}_{\leq X}\\\deg(f)\ \text{even}} } a(f) = \frac{1}{2 \pi i} \oint_{|u|=r} \bigg(\sum_{f \in \mathcal{M}} a(f) u^{\deg(f)}\bigg)\frac{du}{u^{X}(1-u^2)}$$
for the sums over $f$ and $h$. We write $V=V_1V_2^2$ with $V_1$ being a square-free polynomial and $V_2\in\mathcal{M}$, $\text{rad}(C_2)/(\text{rad}(C_2) ,\Delta)=C=AB$, and replace $f,h$ by $Af$ and $Bh$, respectively. We then see that 
\begin{align}\label{S<>0}
S(V\ne0)&=q^{2g+1} \overline{\tau(q)}\sum_{\substack{c_2\leq Z\\c_1+2c_2\leq 2g+1}}\sum_{\substack{C_1\in\mathcal{M}_{c_1}\\C_2\in\mathcal{M}_{c_2}\\C_1 | \Delta,\,C=AB}}  \frac{ \mu(C_1) \chi_{C_1}(NAB)}{|N||A|^{3/2+\alpha} |B|^{3/2 + \beta}|C_1||C_2|^2}\nonumber\\
& \times \frac{1}{(2\pi i)^3}\oint_{|u|=r}\oint_{|v|=r}\oint_{|w|=r} \sum_{V_1\in\mathcal{H}}\, \mathcal{B}(N,A,B,C_1,V_1;u/w,v/w,w,\alpha,\beta)\\
&  \times \frac{(1+uv)dudvdw}{u^{X-\deg(A)+1}v^{Y-\deg(B)+1}w^{\deg(NAB)-\deg(V_1)+c_1+2c_2-2g-1}(1-u^2)(1-v^2)}\nonumber
\end{align}
for any $r<1$, where $\mathcal{B}(N,A,B,C_1,V_1;u,v,w,\alpha,\beta)$ equals
\[
\sum_{f,h,V_2\in\mathcal{M}} \frac{ \chi_{C_1}(fh)\lambda_1(Af) \lambda_2(Bh)u^{\deg(f)}v^{\deg(h)}w^{2\deg(V_2)}G(V_1V_2^2,NABfh)}{ |f|^{3/2+\alpha} |h|^{3/2 + \beta}}.
\]
To proceed we need to study the function $\mathcal{B}(N,A,B,C_1,V_1;u,v,w,\alpha,\beta)$.

\begin{lemma}\label{B}
The function $\mathcal{B}(N,A,B,C_1,V_1;u,v,w,\alpha,\beta)$ defined above may be written as
\[
\mathcal{D}_2(N,A,B,C_1,V_1;u,v,w,\alpha,\beta)\mathcal{L}\Big(E_1\otimes\chi_{C_1V_1},\frac{u}{q^{1+\alpha}}\Big)\mathcal{L}\Big(E_2\otimes\chi_{C_1V_1},\frac{v}{q^{1+\beta}}\Big),
\]
where $\mathcal{D}_2(N,A,B,C_1,V_1;u,v,w,\alpha,\beta)$ is some Euler product which is uniformly convergent provided that $|u|,|v|\leq q^{1/2-2/g}$, $|w|\leq q^{-1/2-\varepsilon}$, and satisfies
\begin{align*}
&\mathcal{D}_2(N,A,B,C_1,V_1;u,v,w,\alpha,\beta)\ll g^{10}\tau_{10}(AB)\sqrt{AB}
\end{align*}
uniformly in this region.
\end{lemma}
\begin{proof}
It is easy to see that $\mathcal{B}(N,A,B,C_1,V_1;u,v,w,\alpha,\beta)$ converges absolutely if $|u|,|v|\leq q^{-\varepsilon}$ and $|w|\leq q^{-1/2-\varepsilon}$. We claim that the sum over $f,h$ and $V_2$ is triply multiplicative. Indeed, one can easily see that the double sum over $f,h$ is multiplicative, so
\begin{align*}
\sum_{f,h\in\mathcal{M}}  & \frac{ \chi_{C_1}(fh)\lambda_1(Af) \lambda_2(Bh)u^{\deg(f)}v^{\deg(h)}G(V_1V_2^2,NABfh)}{ |f|^{3/2+\alpha} |h|^{3/2 + \beta}} \\
&= \prod_{P}\bigg( \sum_{i,j} \frac{\chi_{C_1}(P^{i+j}) \lambda_1(P^{i+a_P}) \lambda_2(P^{j+b_P})u^{i\deg(P)}v^{j\deg(P)} G(V_1 V_2^2,P^{i+j+a_P+b_P+n_P})}{|P|^{(3/2+\alpha)i+(3/2 + \beta)j}}\bigg),
\end{align*}
where $a_P,b_P$ and $n_P$ denote the orders of $A, B$ and $N$ with respect to $P$ respectively. Let $A_P(V_2)$ denote the Euler product above. Note that when $P \nmid V_2$, we have $A_P(V_2)=A_P(1)$.  Then we rewrite the double sum over $f,h$ as 
$$ \prod_P A_P(1) \prod_{P|V_2} \frac{A_P(V_2)}{ A_P(1)}.$$
We introduce the sum over $V_2$ and use the observation that for $(V_2,V_3)=1$ and $P \nmid V_3$ we have $A_P(V_2V_3)= A_P(V_2)$. Then
\begin{align*}
& \prod_P A_P(1)   \sum_{V_2 \in \mathcal{M}} w^{2 \deg(V_2)} \prod_{P|V_2}\frac{ A_P(V_2)}{ A_P(1)} = \prod_P A_P(1) \prod_P \bigg( 1+ \frac{1}{A_P(1)} \sum_{k} A_P(P^{k}) w^{2k \deg(P)} \bigg) \\
&= \prod_{P} \bigg(\sum_{i,j,k}\frac{\chi_{C_1}(P^{i+j}) \lambda_1(P^{i+a_P}) \lambda_2(P^{j+b_P})u^{i\deg(P)}v^{j\deg(P)}w^{2k\deg(P)}G(V_1P^{2k},P^{i+j+a_P+b_P+n_P})}{|P|^{(3/2+\alpha)i+(3/2 + \beta)j}}\bigg),
\end{align*}
and hence the generating series for $f,h,V_2$ is indeed triply multiplicative.

 Now we rewrite $\mathcal{B}(N,A,B,C_1,V_1;u,v,w,\alpha,\beta)$ as
\begin{align*}
\prod_{P\nmid C_1} & \bigg(\sum_{i,j,k}\frac{\chi_{C_1}(P^{i+j}) \lambda_1(P^{i+a_P}) \lambda_2(P^{j+b_P})u^{i\deg(P)}v^{j\deg(P)}w^{2k\deg(P)}G(V_1P^{2k},P^{i+j+a_P+b_P+n_P})}{|P|^{(3/2+\alpha)i+(3/2 + \beta)j}}\bigg)\\
&\times \prod_{P| C_1}\bigg(\sum_{k} \lambda_1(P^{a_P}) \lambda_2(P^{b_P})w^{2k\deg(P)}G(V_1P^{2k},P^{a_P+b_P+n_P})\bigg).
\end{align*}

We next compute the Euler factors at an irreducible $P$ in the region $|u|,|v|\leq q^{1/2-2/g}$, $|w|\leq q^{-1/2-\varepsilon}$. Note that in this region, $w^{2k\deg(P)}\ll_\varepsilon |P|^{-1-\varepsilon}$ if $k\geq1$.

Consider first the case when $P\nmid NABC_1V_1$. The contribution of such an Euler factor is
\begin{align*}
\sum_{i,j,k}\frac{\chi_{C_1}(P^{i+j}) \lambda_1(P^i) \lambda_2(P^j)u^{i\deg(P)}v^{j\deg(P)}w^{2k\deg(P)}G(V_1P^{2k},P^{i+j})}{ |P|^{(3/2+\alpha)i+(3/2 + \beta)j}}.
\end{align*}
In view of Lemma \ref{Gauss},  this is equal to
\[
1+\frac{\chi_{C_1V_1}(P)\lambda_1(P)u^{\deg(P)}}{|P|^{1+\alpha}}+\frac{\chi_{C_1V_1}(P)\lambda_2(P)v^{\deg(P)}}{|P|^{1+\beta}}+O\Big(\frac{1}{|P|^{1+\varepsilon}}\Big),
\]
which justifies the two $L$-functions.

In the case $P|V_1$ but $P\nmid NABC_1$, the Euler factor equals
\begin{align*}
&\sum_{i,j}\frac{\chi_{C_1}(P^{i+j}) \lambda_1(P^i) \lambda_2(P^j)u^{i\deg(P)}v^{j\deg(P)}G(V_1,P^{i+j})}{ |P|^{(3/2+\alpha)i+(3/2 + \beta)j}}+O\Big(\frac{1}{|P|^{1+\varepsilon}}\Big)\\
&\qquad=1-\frac{\lambda_1(P^2)u^{2\deg(P)}}{|P|^{2+2\alpha}}-\frac{\lambda_1(P)\lambda_2(P)(uv)^{\deg(P)}}{|P|^{2+\alpha+\beta}}-\frac{\lambda_2(P^2)v^{2\deg(P)}}{|P|^{2+2\beta}}+O\Big(\frac{1}{|P|^{1+\varepsilon}}\Big)\\
&\qquad=1+O\Big(\frac{1}{|P|^{1+2/g}}\Big).
\end{align*}
Similarly, the corresponding Euler factor is
\[
= \begin{cases}
O\big(|P|^{1/2}\big) & \text{if}\ P|AB\ \text{and}\ P\nmid NC_1V_1,\\
O(1) & \text{if}\ P|AB,\ P| V_1\ \text{and}\ P\nmid NC_1,\\
1+O\big(\frac{1}{|P|^{1+\varepsilon}}\big) & \text{if}\ P|C_1\ \text{and}\ P\nmid NAB,\\
O\big(|P|^{1/2}\big) & \text{if}\ \ P|C_1,\ P|AB\ \text{and}\ P\nmid NV_1,\\
0 & \text{if}\ \ P|C_1,\ P|AB,\ P|V_1\ \text{and}\ P\nmid N,
\end{cases}
\]
and is
\[
\ll_\varepsilon \sum_{i+j\leq 2k} |P|^{1-(1+\varepsilon)k-(i+j)/2}\ll_\varepsilon |P|,
\]
if $P|N$. 

The lemma easily follows by combining these estimates.

\end{proof}

We now return to \eqref{S<>0}. In view of Lemma \ref{B}, we take $r\leq q^{-1/2-\varepsilon}$ and move the $u$ and $v$ contours to $|u|=|v|=rq^{1/2-2/g}$. This creates no poles. Then, by the above result,
\begin{align*}
&\mathcal{B}(N,A,B,C_1,V_1;u/w,v/w,w,\alpha,\beta)\\
&\qquad\qquad \ll g^{10}\tau_{10}(AB)\sqrt{AB}\,\Big|\mathcal{L}\Big(E_1\otimes\chi_{C_1V_1},\frac{u/w}{q^{1+\alpha}}\Big)\mathcal{L}\Big(E_2\otimes\chi_{C_1V_1},\frac{v/w}{q^{1+\beta}}\Big)\Big|\\
&\qquad\qquad \ll g^{10}\tau_{10}(AB)\sqrt{AB}\,\bigg(\Big|\mathcal{L}\Big(E_1\otimes\chi_{C_1V_1},\frac{u/w}{q^{1+\alpha}}\Big)\Big|^2+\Big|\mathcal{L}\Big(E_2\otimes\chi_{C_1V_1},\frac{v/w}{q^{1+\beta}}\Big)\Big|^2\bigg).
\end{align*}
So
\begin{align*}
S(V\ne0)&\ll q^{2g-(X+Y)/2}g^{10}\sum_{\substack{c_2\leq Z\\c_1+2c_2\leq 2g+1}}\sum_{\substack{C_1\in\mathcal{M}_{c_1}\\C_2\in\mathcal{M}_{c_2}\\C_1 | \Delta}}  \frac{\tau(C)\tau_{10}(C)\tau(C_1)}{\sqrt{|C|}|C_1||C_2|^2}\oint_{|u|=q^{1/2-1/g}}\oint_{|v|=q^{1/2-1/g}}\nonumber\\
& \times \oint_{|w|=r}\sum_{V_1\in\mathcal{H}}\, \Big|\mathcal{L}\Big(E_1\otimes\chi_{C_1V_1},\frac{u}{q^{1+\alpha}}\Big)\Big|^2\, \frac{dudvdw}{|w|^{X+Y+\deg(N)-\deg(V_1)+c_1+2c_2-2g+1}}.
\end{align*}
If $\deg(V_1)\leq X+Y+c_1+2c_2-2g$, then we move the $w$ contour to $|w|=q^{-3/4}$, otherwise we move the $w$ contour to $|w|=q^{-5/4}$. Using the upper bounds for moments as in Theorem \ref{thm-ub} (see Remark \ref{rem1}) we find that
\begin{align*}
S(V\ne0)&\ll q^{(X+Y)/2}g^{11}\sum_{\substack{c_2\leq Z\\c_1+2c_2\leq 2g+1}}\sum_{\substack{C_1\in\mathcal{M}_{c_1}\\C_2\in\mathcal{M}_{c_2}\\C_1 | \Delta}}  \frac{\tau(C)\tau_{10}(C)}{\sqrt{|C|}}\\
&\ll q^{(X+Y)/2}g^{11}\sum_{C_2\in\mathcal{M}_{\leq Z}}  \frac{\tau(C_2)\tau_{10}(C_2)}{\sqrt{|\text{rad}(C_2)|}}\ll q^{(X+Y+Z)/2}g^{30},
\end{align*}
which finishes the proof of Lemma \ref{lemmaV<>0}.
\end{proof}
Proposition \ref{prop1} follows upon combining the estimates and choosing $Z=g-(X+Y)/4$.

\section{Upper bounds for moments}\label{upperbounds}

The aim of this section is to bound the tails of the Dirichlet series in the approximate functional equations in Lemmas \ref{fe} and \ref{fe'}. We start with the following upper bounds for moments. 
\begin{theorem} \label{thm-ub}
Let $k>0$, $u = e^{i \theta}, v=e^{i  \gamma}$ with $\theta, \gamma \in [0, 2\pi]$ and let $m = \deg \big(\mathcal{L}(E \otimes \chi_{D}, w)\big)$. Then
\begin{equation*}
 \sum_{D \in \hstar}  \bigg|  \mathcal{L} \Big( E \otimes \chi_{D}, \frac{u}{q^{1/2+\alpha}} \Big)\mathcal{L} \Big( E \otimes \chi_{D}, \frac{v}{q^{1/2+\beta}} \Big) \bigg|^k \ll_\varepsilon q^{2g} g^{\varepsilon} \exp \Big(k \mathcal{M}(u,v,m)+ \frac{k^2}{2} \mathcal{V}(u,v,m) \Big)
\end{equation*} uniformly for $|\alpha|, |\beta| \leq \frac{1}{g}$, where $\mathcal{M}(u,v,m)$ and $\mathcal{V}(u,v,m)$ are given by equations \eqref{mean_formula} and \eqref{variance} respectively.
\end{theorem}

\begin{remark} \label{rem1}
\emph{Note that the same upper bound as above holds if we replace $\mathcal{L}( E \otimes \chi_D, w)$ with $\mathcal{L}(E \otimes \chi_{D\ell},w)$ for a fixed polynomial $\ell$ with $(\ell, \Delta)=1$. Since the proof of the upper bound for this twisted moment is the same as the proof of Theorem \ref{thm-ub}, we only focus on $\ell=1$.}
\end{remark}
We first need the following proposition, whose proof is similar to the proof of Theorem $3.3$ in \cite{altug}. 

\begin{proposition} \label{prop-ub}
Let $D \in \hstar$ and let $m = \deg \big(\mathcal{L}(E \otimes \chi_{D},w)\big)$. Then for $h \leq m$ and $z$ with $\Re(z) \geq 0$ we have
\begin{align*}
\log \big| L(E \otimes \chi_{D}, \tfrac{1}{2}+z) \big| \leq \frac{m}{h} + \frac{1}{h} \Re \bigg(  \sum_{\substack{j \geq 1 \\ \deg(P^j) \leq h}} \frac{ \chi_{D}(P^j) \big( \alpha(P)^j + \beta(P)^j\big)\log q^{h - j \deg(P)}}{|P|^{j\big( \tfrac{1}{2}+z+\tfrac{1}{h \log q} \big)} \log q^j} \bigg).
\end{align*}
\end{proposition}
\begin{proof}

We write $$L(E \otimes \chi_D,s)= \prod_{j=1}^m (1-\alpha_j q^{1/2-s}),$$ where $|\alpha_j|=1$ (see \cite{ger}). Then
\begin{equation}
  \frac{L'}{L}(E \otimes \chi_D, s) =\log q \bigg( -\frac{m}{2}+ \sum_{j=1}^m \Big( \frac{1}{1-\alpha_j q^{1/2-s}} - \frac{1}{2} \Big) \bigg).
  \label{logder}
  \end{equation} 
 
We put $s= \sigma+z$ and integrate equation \eqref{logder} with respect to $\sigma$ from $1/2$ to $\sigma_0$, where $\sigma_0>1/2$. Taking real parts gives
\begin{align*}
&\log \big|L(E \otimes \chi_D, \tfrac{1}{2}+z)\big| - \log \big| L(E \otimes \chi_D, \sigma_0+z)\big| \\
&\quad= \frac{m \log q}{2} \Big(\sigma_0-\frac{1}{2}\Big)  - \frac{1}{2 \log q} \sum_{j=1}^{m} \log \frac{q^{\sigma_0-1/2}-2q^{-\Re(z)} \cos\big( \theta_j - \log q \Im(z)\big) + q^{1/2- \sigma_0-2 \Re(z)}}{1+q^{-2\Re(z)}-2q^{-\Re(z)}\cos\big(\theta_j- \log q \Im(z)\big)},
\end{align*}
where $\alpha_j=e^{i \theta_j}$. We use the inequality $\log(1+x) \geq x/(1+x)$ for $x>0$ and get that
\begin{align}
&\log \big|L(E \otimes \chi_D, \tfrac{1}{2}+z)\big| - \log \big| L(E \otimes \chi_D, \sigma_0+z)\big| \nonumber\\
&\leq \frac{m \log q}{2} \Big(\sigma_0-\frac{1}{2}\Big)-\frac{1}{2 \log q} \sum_{j=1}^m  \frac{(1-q^{1/2-\sigma_0-2\Re(z)})(1-q^{1/2-\sigma_0})}{1-2q^{1/2-\sigma_0-\Re(z)} \cos\big( \theta_j - \log q \Im(z)\big) + q^{1- 2\sigma_0-2 \Re(z)}} \nonumber \\
&\quad = \frac{m \log q}{2} ( \sigma_0 -\frac{1}{2} ) - \frac{G_z(\sigma_0) (1-q^{1/2-\sigma_0-2\Re(z)})(1-q^{1/2-\sigma_0})}{(1-q^{1-2\sigma_0-2\Re(z)}) \log q}, \label{ineq1}
\end{align}
where
\begin{equation}
G_z(\sigma)= \frac{1-q^{1-2\sigma-2 \Re(z)}}{2} \sum_{j=1}^m \frac{1}{1-2q^{1/2-\sigma-\Re(z)} \cos\big( \theta_j - \log q \Im(z)\big) + q^{1- 2\sigma-2 \Re(z)}} .
\label{gz}
\end{equation}

Now similarly as in \cite{altug} we compute the integral
$$ \frac{1}{2 \pi i} \int_2^{2+ 2 \pi i/\log q} - \frac{L'}{L}(E \otimes \chi_D, \sigma+z+w) \frac{q^{hw} q^{-w}}{(1-q^{-w})^2} \, dw$$
in two different ways. First we write 
$$ - \frac{L'}{L}(E \otimes \chi_D, s)= \sum_{n\geq0} \frac{\lambda_D(n)}{q^{ns}},$$ and integrate term by term. Secondly we continue analytically to the left and pick up the residues. We integrate with respect to $\sigma$ from $\sigma_0$ to $\infty$ and take real parts, which gives
\begin{align}\label{ineq2}
&-\frac{1}{ (\log q)^2} \Re \bigg( \sum_{n \leq h} \frac{ \lambda_D(n) \log q^{h-n}}{q^{n(\sigma_0+z)} \log q^n}\bigg) = -\frac{h}{\log q} \log \big| L(E \otimes \chi_D,\sigma_0+z)\big|  \\
& \qquad - \frac{1}{(\log q)^2} \Re \Big( \frac{L'}{L}(E \otimes \chi_D, \sigma_0+z)\Big)+ \sum_{j=1}^m \Re\Big( \int_{\sigma_0}^{\infty} \frac{ (\alpha_j q^{1/2-\sigma_0-z})^h \alpha_j^{-1} q^{\sigma+z-1/2}}{(1-\alpha_j^{-1} q^{\sigma+z-1/2})^2} \, d \sigma \Big). \nonumber
\end{align} 
Now we have
\begin{align*}
&\Re\Big( \int_{\sigma_0}^{\infty} \frac{ (\alpha_j q^{1/2-\sigma_0-z})^h \alpha_j^{-1} q^{\sigma+z-1/2}}{(1-\alpha_j^{-1} q^{\sigma+z-1/2})^2} \, d \sigma \Big) \leq \int_{\sigma_0}^{\infty} \frac{ q^{(1/2-\sigma-\Re(z))h} q^{\sigma+\Re(z)-1/2}}{ |1-\alpha_j^{-1} q^{\sigma+z-1/2}|^2} \, d \sigma \\
&\qquad\qquad= \int_{\sigma_0}^{\infty} \frac{q^{(1/2-\sigma-\Re(z))h} q^{1/2 - \sigma-\Re(z)}}{1-2 q^{1/2- \sigma - \Re(z)} \cos\big(\theta_j-\log q \Im(z)\big)+q^{1-2 \sigma- 2 \Re(z)}} \, d \sigma.
\end{align*}
By taking the derivative of $$f(x) = \frac{q^{-x}}{1-2 q^{1/2-x-\Re(z)} \cos\big(\theta_j - \log q \Im(z)\big)+q^{1-2x-2\Re(z)}}$$ we can see that $f$ is decreasing on $[\sigma_0, \infty)$.
Hence
\begin{align*}
&\sum_{j=1}^m \Re\Big( \int_{\sigma_0}^{\infty} \frac{ (\alpha_j q^{1/2-\sigma_0-z})^h \alpha_j^{-1} q^{\sigma+z-1/2}}{(1-\alpha_j^{-1} q^{\sigma+z-1/2})^2} \, d \sigma \Big)   \\
&\qquad  \leq \sum_{j=1}^m\frac{q^{1/2 - \sigma_0 - \Re(z)}}{1- 2q^{1/2-\sigma_0 - \Re(z)} \cos\big( \theta_j - \log q \Im(z)\big)+q^{1-2 \sigma_0 - 2\Re(z) }} \int_{\sigma_0}^{\infty}q^{(1/2-\sigma-\Re(z))h} \, d \sigma \\ 
&\qquad= \frac{2 G_z(\sigma_0) q^{(1/2- \sigma_0 - \Re(z))h} q^{1/2-\sigma_0-\Re(z)}}{(1-q^{1-2\sigma_0-2\Re(z)})\log q^h} \leq  \frac{2 G_z(\sigma_0) q^{(1/2-\sigma_0 - \Re(z))h}}{ (\sigma_0+\Re(z) -1/2)\log q  \log q^h },
\end{align*}
with $G_z(\sigma_0)$ as in equation \eqref{gz}. Now from equation \eqref{logder} note that
$$ \Re \Big(\frac{L'}{L} (E \otimes \chi_D, \sigma_0+z) \Big)= \log q\Big( - \frac{m}{2} + G_z(\sigma_0)\Big).$$
Combining the equations above and \eqref{ineq2} gives
\begin{align*}
&\log \big|L(E \otimes \chi_D, \sigma_0+z)\big| \\
&\qquad\qquad \leq \frac{1}{\log q^h} \Re \Big( \sum_{n \leq h} \frac{ \lambda_D(n) \log q^{h-n}}{q^{n(\sigma_0+z)} \log q^n}\Big) + \frac{1}{h} \Big( \frac{m}{2} - G_z(\sigma_0)\Big)+ \frac{2 G_z(\sigma_0) q^{(1/2-\sigma_0 - \Re(z))h}}{ h   (\sigma_0+\Re(z) -1/2)\log q^h}.
\end{align*}
This and \eqref{ineq1} lead to
\begin{align*}
\log \big| L(&E \otimes \chi_D, \tfrac{1}{2}+z)\big| \leq \frac{m}{2h}+ \frac{m \log q}{2} \Big(\sigma_0-\frac{1}{2}\Big)+  \frac{1}{\log q^h} \Re \Big( \sum_{n \leq h} \frac{ \lambda_D(n) \log q^{h-n}}{q^{n(\sigma_0+z)} \log q^n}\Big)\\
&+ G_z(\sigma_0) \Big( -\frac{1}{h} + \frac{2  q^{(1/2-\sigma_0 - \Re(z))h}}{ h  (\sigma_0+\Re(z) - 1/2) \log q^h} - \frac{(1-q^{1/2- \sigma_0 - 2 \Re(z)})(1-q^{1/2-\sigma_0})}{( 1-q^{1-2\sigma_0-2 \Re(z)})\log q}  \Big).
\end{align*}
Choosing $\sigma_0 = 1/2 + 1/\log q^h$ ensures that the coefficient of $G_z(\sigma_0)$ is negative. Since
$$ \lambda_D(n) = \log q \sum_{j|n} \sum_{\deg(P^j) =n} \deg(P) \chi_D(P)^j \big( \alpha(P)^j+\beta(P)^j\big),$$
the conclusion follows.
%have
%$$- \Re  \frac{L'}{L} (E \otimes \chi_D, \sigma+z) = \log q \Big( \frac{k}{2} -  \Re F(\sigma+z)\Big).$$
%Now we integrate from $1/2$ to $\sigma_0$ with respect to $\sigma$, where $\sigma_0>1/2$.
\end{proof}
Before proving Theorem \ref{thm-ub}, we also need the following lemma (see Lemma $8.4$ in \cite{fl4}).
\begin{lemma}
Let $h,l$ be integers such that $2hl \leq 2g+1$. For any complex numbers $a(P)$ we have
$$ \sum_{D \in \hstar} \Big| \sum_{\deg(P) \leq h} \frac{\chi_D(P) a(P)}{|P|^{1/2}} \Big|^{2l} \ll q^{2g} \frac{ (2l)!}{l! 2^l} \Big(\sum_{\deg(P) \leq h} \frac{ |a(P)|^2}{|P|} \Big)^l.$$
\label{dir_mom}
\end{lemma}
Let $$N(V,u,v) = \bigg| \Big\{D \in \hstar : \log \Big| \mathcal{L}\Big(E \otimes \chi_{D}, \frac{u}{q^{1/2+\alpha}}\Big) \mathcal{L}\Big(E \otimes \chi_{D}, \frac{v}{q^{1/2+\beta}}\Big)\Big| \geq V + \mathcal{M}(u,v,m) \Big\} \bigg|.$$
We will prove the following lemma.
\begin{lemma} \label{ub_n}
If $\sqrt{\log m} \leq V \leq \mathcal{V}(u,v,m)$, then
$$N(V,u,v) \ll q^{2g+1} \exp \bigg( -\frac{V^2}{2 \mathcal{V}(u,v,m) } \Big(1-\frac{12}{\log \log m} \Big)\bigg) ; $$
if $ \mathcal{V}(u,v,m)< V \leq  \frac{\log \log m}{13}\mathcal{V}(u,v,m)$, then
$$ N(V,u,v) \ll q^{2g+1} \exp \bigg( -\frac{V^2}{2 \mathcal{V}(u,v,m) } \Big( 1- \frac{12V}{ \mathcal{V}(u,v,m) \log \log m} \Big)^2 \bigg);$$
if $V > \frac{\log \log m}{13}\mathcal{V}(u,v,m)$, then
$$N(V,u,v) \ll q^{2g+1} \exp \Big(-\frac{V \log V}{4500} \Big).$$
\end{lemma}

Using Lemma \ref{ub_n} above we can prove Theorem \ref{thm-ub} as follows.
\begin{proof}[Proof of Theorem \ref{thm-ub}]
We have the following.
\begin{align*}
\sum_{D \in \hstar} \bigg| \mathcal{L}&\Big(E \otimes \chi_{D},  \frac{u}{q^{1/2+\alpha}}\Big)\mathcal{L}\Big(E \otimes \chi_{D}, \frac{v}{q^{1/2+\beta}}\Big ) \bigg|^k \\&= - \int_{-\infty}^{\infty} \exp \big(kV + k \mathcal{M}(u,v,m)\big) dN(V,u,v) \\
&= k \int_{-\infty}^{\infty} \exp \big(kV+k \mathcal{M}(u,v,m)\big) N(V,u,v) dV.
\end{align*}
In the equation above we use Lemma \ref{ub_n} in the form
$$N(V,u,v) \ll 
\begin{cases}
q^{2g+1} m^{o(1)} \exp \Big(-\frac{V^2}{2 \mathcal{V}(u,v,m)} \Big) & \mbox{ if } V \leq 8k \mathcal{V}(u,v,m), \\
q^{2g+1} m^{o(1)} \exp(-4kV) & \mbox{ if }V > 8k \mathcal{V}(u,v,m),
\end{cases}
$$
which finishes the proof of Theorem \ref{thm-ub}.
\end{proof}
\begin{proof}[Proof of Lemma  \ref{ub_n}]
%
%Let $m= \deg \mathcal{L} (E \otimes \chi_{D},u)$. 
We assume without loss of generality that $\alpha, \beta$ are positive and real. Indeed, notice that if $\alpha \in \mathbb{C}$, since $|\alpha| \leq \tfrac{1}{g}$, we have $|\Re(\alpha) | \leq \tfrac{1}{g}$ and $| \Im(\alpha) | \leq \frac{1}{g}$.  The proof that follows goes through in exactly the same way, with $\alpha$ replaced by $\Re(\alpha)$ and $\theta$ replaced by $\theta- \Im(\alpha)  \log q$. Once we assume $\alpha$ is real, we can also assume that $\alpha$ is positive, since by the functional equation we have
%Otherwise note that by the functional equation, we have 
$$ \Big| \mathcal{L}\big( E \otimes \chi_D, \frac{u}{q^{\frac{1}{2}+\alpha}} \big) \Big|= q^{-\alpha(\mathfrak{n}+2 \deg(D))} \Big| \mathcal{L} \big(E \otimes \chi_D, \frac{1}{q^{\frac{1}{2}-\alpha}u} \big) \Big|.$$ 
Let $$\frac{m}{h} = \frac{V}{A}\qquad \text{and}\qquad h_0 = \frac{h}{\log m},$$ where
$$ A= 
\begin{cases}
\displaystyle \frac{\log \log m}{2} & \mbox{ if } V \leq \mathcal{V}(u,v,m), \\
\displaystyle \frac{\log \log m}{2V}\mathcal{V}(u,v,m)  & \mbox{ if } \mathcal{V}(u,v,m) < V \leq \frac{ \log \log m}{13}\mathcal{V}(u,v,m), \\
 \frac{13}{12} & \mbox{ if } V> \frac{ \log \log m}{13}\mathcal{V}(u,v,m).
\end{cases}
$$
Using Proposition \ref{prop-ub} gives
\begin{align}\label{ineq_prop}
& \log \bigg| \mathcal{L}\Big(E \otimes \chi_{D},  \frac{u}{q^{1/2+\alpha}}\Big)\mathcal{L}\Big(E \otimes \chi_{D}, \frac{v}{q^{1/2+\beta}}\Big ) \bigg| \leq \frac{2m}{h} \\
& +\frac{1}{h} \Re \bigg(  \sum_{\substack{ j \geq 1 \\ \deg(P^j) \leq h}} \frac{\chi_{D}(P^j) ( \alpha(P)^j + \beta(P)^j)\log q^{h-j \deg(P)}}{|P|^{(1/2+1/h \log q)j}\log q^j}  \big(|P|^{-(\alpha-i\theta/ \log q)j} + |P|^{-(\beta- i\gamma/ \log q)j}\big)\bigg). \nonumber
\end{align}
Note that the contribution of the terms with $j \geq 3$ is bounded by $O(1)$.

The terms with $j=2$ in \eqref{ineq_prop} will contribute, up to a term of size $O(\log \log m)$ coming from those $P$ with $P|D$,
\begin{align*}
\frac{1}{2 } \sum_{\deg(P) \leq h/2} \frac{ (h - 2\deg(P))(\lambda (P^2)-1) }{h |P|^{1+ 2/h \log q}} \Big(\frac{ \cos(2 \theta \deg(P))}{|P|^{2 \alpha}} + \frac{ \cos(2 \gamma \deg(P))}{|P|^{2 \beta}} \Big).
\end{align*}
Let 
$$ F_{\alpha}(h, \theta) = \sum_{n=1}^h \frac{ \cos(2n \theta)}{n q^{2n \alpha +n/h \log q}}.$$
Similarly as in \cite{fl4} (Lemma $9.1$), we can show that
\begin{equation}
F_{\alpha} (h, \theta) =  \log \min  \Big \{ h ,\frac{1}{ \overline{2 \theta}} \Big\} +O(1),
\label{f_expr}
\end{equation}
where for $\theta \in [0, 2\pi]$ we denote $\overline{\theta} = \min \{ \theta, 2 \pi - \theta \}$. 
Now using the fact that
\begin{equation}
 \sum_{\deg(P) \leq h} \frac{ \lambda(P^2) \cos(2 \theta \deg(P))}{|P|^{1+1/h\log q}} = O(\log \log h),
 \label{ub_pf}
 \end{equation}

it follows that the contribution from $j=2$ will be equal to
\begin{align}
-\frac{1}{2 } & \sum_{\deg(P) \leq \frac{h}{2}} \frac{h - 2 \deg(P)}{h |P|^{1+ 2/h \log q}} \Big(\frac{ \cos(2 \theta \deg(P))}{|P|^{2 \alpha}} + \frac{ \cos(2 \gamma \deg(P))}{|P|^{2 \beta}} \Big)+O( \log \log m) \nonumber \\
&=-\frac12\bigg(F_{\alpha}\Big(\frac h2,\theta\Big)+F_{\beta}\Big(\frac h2,\gamma\Big)\bigg) + O(\log \log m) \nonumber\\
&\leq  -\frac12\big(F_{\alpha}(m, \theta) + F_{\beta}(m,\gamma)\big)+\frac{2m}{h}+ O(\log \log m)= \mathcal{M}(u,v,m) + \frac{2m}{h} + O( \log \log m), \label{aux}
\end{align}
where
\begin{equation}
\mathcal{M}(u,v,m) = -\frac{1}{2 } \bigg( \log \min  \Big \{ m ,\frac{1}{ \overline{2\theta}} \Big\}+ \log \min  \Big \{ m ,\frac{1}{ \overline{2 \gamma}} \Big\} \bigg)
\label{mean_formula}
\end{equation} by formula \eqref{f_expr}. Note that in the second line of the equation above we used the fact that
\begin{align*}
F_{\alpha}(m,\theta) - F_{\alpha}\Big(\frac h2,\theta\Big) &= \sum_{n=1}^m \frac{ \cos(2n \theta)}{n q^{2n \alpha} e^{n/m}}- \sum_{n=1}^{h/2} \frac{\cos(2n \theta)}{n q^{2n \alpha} e^{2n/h}},
\end{align*} and since $e^{-x} = 1+O(x)$, it follows that
\begin{align*}
F_{\alpha}(m,\theta) - F_{\alpha}\Big(\frac h2,\theta\Big)= \sum_{n=h/2+1}^m \frac{\cos(2n \theta)}{nq^{2n \alpha}} +O(1) \leq \frac{2m}{h} + O(1).
\end{align*}
\kommentar{Note that in the second line of the equation above we used the fact that
\begin{align*}
F_{\alpha}(m,\theta) - F_{\alpha}\Big(\frac h2,\theta\Big) &= \sum_{n=1}^m \frac{ \cos(2n \theta)}{n q^{2n \alpha} e^{n/m}}- \sum_{n=1}^{h/2} \frac{\cos(2n \theta)}{n q^{2n \alpha} e^{2n/h}} = \sum_{n=h/2+1}^m \frac{\cos(2n \theta)}{nq^{2n \alpha}} +O(1) \leq \frac{2m}{h} + O(1).
\end{align*}}
Hence, using equation \eqref{aux} in \eqref{ineq_prop} we get
\begin{align*}
& \log \bigg| \mathcal{L}\Big(E \otimes \chi_{D},  \frac{u}{q^{1/2+\alpha}}\Big)\mathcal{L}\Big(E \otimes \chi_{D}, \frac{v}{q^{1/2+\beta}}\Big ) \bigg| \leq \mathcal{M}(u,v,m) +\frac{5m}{h}\\
&\qquad\qquad +  \sum_{\deg(P) \leq h} \frac{  (h - \deg(P))\chi_{D}(P) \lambda(P)}{h |P|^{\tfrac{1}{2}+\tfrac{1}{h \log q}}} \Big(  \frac{ \cos( \theta \deg(P))}{|P|^{\alpha}} + \frac{ \cos( \gamma \deg(P))}{|P|^{\beta}}\Big) .
\end{align*}

Let $S_1$ be the sum above truncated at $\deg(P) \leq h_0$ and $S_2$ be the sum over primes with $h_0 < \deg(P) \leq h$. If $D$ is such that $$ \log \bigg| \mathcal{L}\Big(E \otimes \chi_{D},  \frac{u}{q^{1/2+\alpha}}\Big)\mathcal{L}\Big(E \otimes \chi_{D}, \frac{v}{q^{1/2+\beta}}\Big ) \bigg| \geq \mathcal{M}(u,v,m)+V,$$ then $$S_1 \geq V_1:= V \Big( 1-\frac{6}{A}\Big) \qquad\text{or}\qquad S_2 \geq \frac{V}{A}.$$ 
Let $$\mathcal{F}_1 = \{ D \in \hstar : S_1 \geq V_1\}\qquad \text{and} \qquad\mathcal{F}_2 = \{ D \in \hstar : S_2 \geq V/A \}.$$

If $D \in \mathcal{F}_2$, then by Markov's inequality and Lemma \ref{dir_mom} it follows that
$$ |\mathcal{F}_2| \ll q^{2g} \frac{ (2l)!}{l! 2^l} \Big( \frac{A}{V}\Big)^{2l} \Big( \sum_{h_0<\deg(P) \leq h} \frac{ |a(P)|^2}{|P|} \Big)^l,$$ 
for any $l \leq g/h$ where $$ a(P) =\frac{  (h- \deg(P)) \lambda(P)}{h |P|^{\tfrac{1}{h \log q}}} \Big( \frac{\cos(\theta \deg(P))}{|P|^{\alpha}} + \frac{ \cos(\gamma \deg(P))}{|P|^{\beta}}\Big).$$
Picking $l = [g/h]$ and noting that $a(P) \ll 1$ and $m=4g+O(1)$, we get that
\begin{equation}
|  \mathcal{F}_2 | \ll q^{2g}\Big(\frac{2l}{e} \Big)^{l} \Big( \frac{A}{V} \Big)^{2l}  (\log \log m)^l \ll q^{2g} \exp \Big(-\frac{V\log V}{8A}  \Big).
\label{f2_bd}
\end{equation}

If $D \in \mathcal{F}_1$ then for any $l \leq g/h_0$, we have
$$ | \mathcal{F}_1| \ll q^{2g} \frac{ (2l)!}{l! 2^l} \frac{1}{V_1^{2l}}\Big( \sum_{\deg(P) \leq h_0} \frac{ |a(P)|^2}{|P|} \Big)^l.$$
Using the expression for $a(P)$ and equation \eqref{ub_pf} we get that 
$$  | \mathcal{F}_1| \ll q^{2g} \frac{ (2l)!}{ l! 2^lV_1^{2l}} \big( \mathcal{V}(u,v,m) + O(\log \log m) \big)^l,$$ where

\begin{align}
  \mathcal{V}(u,v,m) &= \log m + \frac{1}{2}\big(F_{\alpha}(m,\theta)+F_{\beta}(m,\gamma)\big)+ F_{(\alpha+\beta)/2} \Big(m, \frac{\overline{\theta+\gamma}}{2}\Big)+F_{(\alpha+\beta)/2}\Big(m,\frac{ \overline{ \theta- \gamma}}{2} \Big) \nonumber \\
 &= \log m + \frac{1}{2}\bigg( \log \min  \Big \{ m ,\frac{1}{ \overline{2 \theta}} \Big\} +  \log \min  \Big \{ m ,\frac{1}{ \overline{2\gamma}} \Big\}\bigg)\nonumber\\
 &\qquad\qquad +  \log \min  \Big \{ m ,  \frac{1}{ \overline{ \theta+\gamma}} \Big\}+ \log \min  \Big \{ m , \frac{1}{ \overline{ \theta-\gamma}} \Big\},
 \label{variance}
 \end{align}
 and the last line of the equation above follows from equation \eqref{f_expr}. Then
 $$ |\mathcal{F}_1| \ll q^{2g} \Big( \frac{2l}{eV_1^2} \big( \mathcal{V}(u,v,m) + O(\log \log m)\big)\Big)^l.$$
 If $V \leq \mathcal{V}(u,v,m)$, then we pick $l = [ V_1^2/2 \mathcal{V}(u,v,m)]$, and if $V > \mathcal{V}(u,v,m)$, then we pick $l=[10V]$. In doing so we get 
 \begin{equation}
 | \mathcal{F}_1| \ll q^{2g} \exp \Big( - \frac{V_1^2}{2 \mathcal{V}(u,v,m)} \Big)+ q^{2g} \exp ( -V \log V).
 \label{f1_bd}
 \end{equation}
 Combining the bounds \eqref{f1_bd} and \eqref{f2_bd} finishes the proof of Lemma \ref{ub_n}.
\end{proof}

The following is an immediate corollary of Theorem \ref{thm-ub}.
\begin{corollary} \label{cor-upper-bound}
Let $u= e^{i \theta}$ and $v = e^{i \gamma}$ with $\theta, \gamma \in [0,2 \pi ]$. Then
\begin{align*}
\sum_{D \in \hstar}  \bigg|  \mathcal{L} \Big( E \otimes \chi_D, \frac{u}{\sqrt{q}} \Big)\mathcal{L} \Big( E \otimes \chi_D, \frac{v}{\sqrt{q}} \Big) \bigg| &\ll_\varepsilon q^{2g}  g^{\frac{1}{2}+\varepsilon}\min \Big\{ g, \frac{1}{\overline{2\theta}}\Big\}^{-1/4} \min \Big\{ g, \frac{1}{\overline{2\gamma}}\Big\}^{-1/4} \\
 &\times  \min\Big\{g , \frac{1}{ \overline{ (\theta- \gamma)  }}\Big \}^{1/2} \min\Big\{g , \frac{1}{ \overline{ (\theta+ \gamma) }} \Big\}^{1/2}.
 \end{align*}
\end{corollary}

For $N|\Delta^\infty$ and fixed $n\in\mathbb{N}$, we define the truncated sums
\begin{equation}\label{truncation}
\mathcal{E}_{1,E}(N,X,n):=\sum_{X<\deg(f)\leq n+\deg(D)} \frac{ \lambda(f) \chi_D(Nf)}{ \sqrt{|f|}}
\end{equation}
and
\begin{equation}
\label{truncation2}
\mathcal{E}_{2,E}(N,X,n):=\sum_{X<\deg(f)\leq n+\deg(D)} \frac{\big(n+\deg(D)-\deg(f)\big) \lambda(f) \chi_D(Nf)}{ \sqrt{|f|}}.
\end{equation}
We are now ready to prove the following upper bounds for $\mathcal{E}_{i,E}(N,X,n)$. 
\begin{proposition}\label{prop2}
For $i=1,2$ and any fixed $n\in\mathbb{N}$ we have
\[
\sum_{D\in\hstar}\big|\mathcal{E}_{i,E}(N,X,n)\big|^2\ll_{\varepsilon} q^{2g}g^{1/2+\varepsilon}\big(2g-X\big)^{2i}.
\]
\end{proposition}
\begin{proof}
Using the Perron formula for the sum over $f$ in \eqref{truncation}, we get that
\begin{equation*}
\mathcal{E}_{1,E}(N,X,n) =  \frac{\chi_D(N)}{2 \pi i} \oint_{|u|=1} \mathcal{L} \Big( E \otimes \chi_D, \frac{u}{\sqrt{q}} \Big) \Big( \frac{1}{u^{n+\deg(D)}}- \frac{1}{u^X} \Big) \, \frac{du}{u(1-u)}.
\end{equation*}
Note that there is no pole at $u=1$. So we need to bound the following expression
\begin{align} \label{double-integral}
\frac{1}{(2 \pi i)^2} \oint_{|u|=1} \oint_{|v|=1} & \sum_{D\in\hstar}\mathcal{L} \Big( E \otimes \chi_D, \frac{u}{\sqrt{q}} \Big) \mathcal{L} \Big( E \otimes \chi_D, \frac{v}{\sqrt{q}} \Big)\nonumber\\
&\times \Big( \frac{1}{u^{n+\deg(D)}}- \frac{1}{u^X} \Big) \Big( \frac{1}{v^{n+\deg(D)}}- \frac{1}{v^X} \Big) \, \frac{dudv}{uv(1-u)(1-v)}.
\end{align}
We use Corollary \ref{cor-upper-bound} to bound the integral above and consider $\theta$ and $\gamma$ on different arcs on the unit circle. We bound the integral on these arcs and notice that we obtain the biggest upper bound when $\theta, \gamma$ are not close to $0$ (i.e. $u$ and $v$ are not close to $1$) and when $\theta$ is not close to $\gamma$ or to $2 \pi - \gamma$ (by close we mean on an arc of length on the scale of $1/g$). 

For example, if $\theta$ and $\gamma$ are both on an arc $C_1$ of length on the scale of $1/g$ around $0$, then the double integral in \eqref{double-integral} over the arcs $C_1$ is $O_\varepsilon\big(q^{2g} g^{-1+\varepsilon}\big)$ (since from the Corollary we get a power of $g$ which gets multiplied by $g^{-2}$, the product of the sizes of the arcs.)

If $\theta, \gamma$ are both on the complement of $C_1$, but close to each other (i.e. $\theta$ is within $1/g$ of $\gamma$), we get that the corresponding integral over the two arcs is $O_\varepsilon\big(q^{2g} g^{\varepsilon}\big)$. We get a similar bound if $\theta$ is close to $2 \pi -\gamma$, under the same conditions. 

We are left with the case when $\theta,\gamma$ are on the complement of $C_1$ and $\theta$ is far from $\gamma$ and from $2 \pi - \gamma$. In this case the corresponding integral will be $O_\varepsilon\big(q^{2g} g^{1/2+\varepsilon}\big).$ This finishes the proof of the upper bound when $i=1$.

When $i=2$, using the Perron formula for the sum over $f$ in \eqref{truncation2}, we have
\begin{align*}
\mathcal{E}_{2,E}(N,X,n)& =  \frac{\chi_D(N)}{2 \pi i} \oint_{|u|=1} \mathcal{L} \Big( E \otimes \chi_D, \frac{u}{\sqrt{q}} \Big) \\
&\times \Big( \frac{1}{u^{n+\deg(D)}}-\frac{(n+\deg(D)-X)(1-u)+u}{u^{X+1}}\Big) \, \frac{du}{(1-u)^2}.
\end{align*}
Hence
\begin{align*}
 \mathcal{E}_{2,E}(N,X,n)^2  =&\, \frac{1}{(2 \pi i)^2}  \oint_{|u|=1} \oint_{|v|=1}  \mathcal{L} \Big( E \otimes \chi_D, \frac{u}{\sqrt{q}} \Big) \mathcal{L} \Big( E \otimes \chi_D, \frac{v}{\sqrt{q}} \Big)  \nonumber \\
&\times \Big( \frac{1}{u^{n+\deg(D)}}-\frac{(n+\deg(D)-X)(1-u)+u}{u^{X+1}}\Big)\\
&\times \Big( \frac{1}{v^{n+\deg(D)}}-\frac{(n+\deg(D)-X)(1-v)+v}{v^{X+1}}\Big) \, \frac{dudv}{(1-u)^2(1-v)^2}.
\end{align*}
We proceed as before and keeping in mind that $|u|=|v|=1$, it follows that
$$ \sum_{D \in \hstar} | \mathcal{E}_{2,E}(N,X,n)|^2 \ll_\varepsilon q^{2g} g^{1/2+\varepsilon} (2g-X)^4,
$$
as required.
\end{proof}

\section{Proof of Theorem \ref{L^1}}
For $N|\Delta^\infty$, let 
$$ R_E (N,X) := \sum_{D \in \hstar} \sum_{f \in \mathcal{M}_{\leq X}} \frac{ \lambda(f) \chi_D(Nf)}{\sqrt{|f|}}.$$
We will prove the following lemma.
\begin{lemma}\label{lemma_th1}
We have 
\begin{equation*}
 R_E(N,X) = |\mathcal{H}_{2g+1}|\mathcal{C}_E(N;1)   L \big( \emph{Sym}^2 E, 1\big)+O_\varepsilon \big( q^{2g-X/2 + \varepsilon g} \big)+O_\varepsilon\big(q^{X/2+\epsilon g}\big),
 \end{equation*}
 where the value $\mathcal{C}_E(N;1)$ is defined in \eqref{def_bm}.
\end{lemma}
\begin{proof}
Note that
\[
R_E(N,X)=S_{E,E}(N,X,0;0,0),
\]
where $S_{E_1,E_2}(N,X,Y;\alpha,\beta)$ is defined as in \eqref{formulaS}. We proceed as in Section \ref{mainpropS}, see \eqref{formulaS1}, \eqref{formulaS2} and \eqref{formulaS3}, and write
\begin{align*}
R_E(N,X)&=S_{E,E}(N,X,0,g;0,0)\\
&=R_E(N,X;V=0)+R_E(N,X;V\ne 0),
\end{align*}
where $R_E(N,X;V\ne 0)=S_{E,E}(N,X,0,g;0,0;V\ne 0)$ and
\begin{align*}
R_E(N,X;V=0)&=S_{E,E}(N,X,0,g;0,0;V=0)\\
&=|\mathcal{H}_{2g+1}|  \sum_{\substack{f \in \mathcal{M}_{\leq X} \\ Nf= \square}} \frac{ \lambda(f) }{\sqrt{|f|}} \prod_{P | \Delta f} \bigg(1- \frac{1}{|P|} \bigg)\sum_{\substack{C_2 | (\Delta f)^{\infty} \\  \deg(C_2) \leq g}} \frac{1}{|C_2|^2} + O_\varepsilon(q^{\varepsilon g}).
\end{align*}

We first evaluate $R_E(N,X;V=0)$. From \eqref{C2formula} we have
\[
\sum_{\substack{C_2 | (\Delta f)^{\infty} \\  \deg(C_2) \leq g}} \frac{1}{|C_2|^2}=\prod_{P | \Delta f} \bigg(1- \frac{1}{|P|^2} \bigg)^{-1}+O_\varepsilon(q^{-2g+\varepsilon g}).
\]
The contribution of the error term to $R_E(N,X;V=0)$ is
\begin{align*}
\ll_\varepsilon q^{\varepsilon g} \sum_{l \in \mathcal{M}_{\leq X/2}} \frac{|\lambda_f(Nl^2)|}{|l|}  \ll_\varepsilon q^{\varepsilon g}.
\end{align*}
Hence
\begin{align*}
R_E(N,X;V=0)=|\mathcal{H}_{2g+1}|  \sum_{\substack{f \in \mathcal{M}_{\leq X} \\ Nf= \square}} \frac{ \lambda(f) }{\sqrt{|f|}} \prod_{P | \Delta f} \bigg(1+ \frac{1}{|P|} \bigg)^{-1}+ O_\varepsilon(q^{\varepsilon g}).
\end{align*}

Applying the Perron formula to the sum over $f$ yields
\begin{align}\label{RV=0}
R_E(N,X;V=0)=\frac{|\mathcal{H}_{2g+1}| }{2\pi i}\oint_{|u|=r}\mathcal{B}_{E}(N;u)\frac{du}{u^{X+1}(1-u)} + O_\varepsilon(q^{\varepsilon g})
\end{align}
for any $r<1$, where
$$\mathcal{B}_{E}(N;u) = \sum_{\substack{f \in \mathcal{M} \\ Nf = \square}} \frac{ \lambda(f)u^{\deg(f)}}{\sqrt{|f|}}\prod_{P | \Delta f} \bigg(1+ \frac{1}{|P|} \bigg)^{-1}.$$ We can write $\mathcal{B}_{E}(N;u)$ in terms of its Euler product,
\begin{align}\label{def_bm}
\mathcal{B}_{E}(N;u)&= \prod_{P \nmid \Delta} \bigg( 1+\bigg(1+ \frac{1}{|P|} \bigg)^{-1} \sum_{i\geq 1}\frac{ \lambda(P^{2i})  u^{2i \deg(P)} }{|P|^{i}}\bigg)\nonumber\\
&\times \prod_{P|\Delta} \bigg( \bigg(1+ \frac{1}{|P|} \bigg)^{-1}\sum_{i+\text{ord}_P(N)\ \text{even}}\frac{ \lambda(P^i) u^{i \deg(P)}}{|P|^{i/2}} \bigg)\nonumber\\
&=\mathcal{C}_E(N;u)\,\mathcal{L} \Big(\text{Sym}^2 E, \frac{u^2}{q}\Big),
\end{align} where $\mathcal{C}_E(N;u)$ is some Euler product which is uniformly bounded for $|u|\leq q^{1/2-\varepsilon}$. We shift the contour in \eqref{RV=0} to $|u|=q^{1/2-\varepsilon}$, encountering a simple pole at $u=1$. 
Then 
\begin{equation}
R_E(N,X;V=0) =|\mathcal{H}_{2g+1}|\mathcal{C}_E(N;1)   L \big( \text{Sym}^2 E, 1\big)+O_\varepsilon \big( q^{2g-X/2 + \varepsilon g} \big).
\label{main_th1}
\end{equation}

Now we will bound $R_E(N,X;V\ne 0)$. As in Subsection \ref{propV<>0}, see \eqref{formulaS4}, it suffices to bound the term
\begin{align*}
R(V\ne0) =&\,q^{2g+1} \overline{\tau(q)} \sum_{\substack{f \in \mathcal{M}_{\leq X}\\ \deg(Nf)\ \text{odd}}} \frac{ \lambda(f) }{ |N||f|^{3/2} }  \sum_{\substack{C_1 | \Delta\\C_2 | (\Delta f)^{\infty} \\ \deg(C_1)+ 2 \deg(C_2) \leq 2g+1}}  \frac{ \mu(C_1) \chi_{C_1}(Nf)}{|C_1||C_2|^2}\nonumber\\
&\times \sum_{V \in \mathcal{M}_{ \deg(Nf) + \deg(C_1) + 2 \deg(C_2) - 2g-2}}  G(V,Nf).
\end{align*}
Using the fact that
$$\sum_{\substack{C_2 \in \mathcal{M}_{c_2} \\ C_2| (\Delta f)^{\infty}}} \frac{1}{|C_2|^2} = \frac{q^{-2c_2}}{2 \pi i} \oint_{|w|=r}    \prod_{P | \Delta f} \Big(1-w^{\deg(P)}\Big)^{-1}\, \frac{dw}{w^{c_2+1}} $$ for $r<1$, and writing $V=V_1V_2^2$ with $V_1$ a square-free polynomial, we have
\begin{align*}
 R(V\ne0) &= \frac{q^{2g+1}\overline{\tau(q)}}{|N|}\sum_{c_1+2c_2 \leq 2g+1} q^{-2c_2}\sum_{\substack{C_1 \in \mathcal{M}_{c_1} \\ C_1 | \Delta}} \frac{ \mu(C_1) \chi_{C_1}(N)}{|C_1|} \\
&\times  \sum_{\substack{n\leq X \\ n +\deg(N) \text{ odd}}}\sum_{\substack{j\leq n+\deg(N)+c_1+2c_2-2g-2 \\ j+c_1\text{ odd}}} \sum_{V_1 \in \mathcal{H}_j} \sum_{V_2 \in \mathcal{M}_{(n+\deg(N)+c_1-j)/2+c_2-g-1}} \\
& \times  \frac{1}{2 \pi i}\oint_{|w|=r} \sum_{f \in \mathcal{M}_n} \frac{\chi_{C_1}(f)\lambda(f) G(V_1V_2^2,Nf)}{|f|^{3/2} }\prod_{P |\Delta f} \Big(1-w^{\deg(P)}\Big)^{-1} \, \frac{dw}{w^{c_2+1}}.
\end{align*}

Now
\begin{align*}
&\sum_{f \in \mathcal{M}}   \frac{\chi_{C_1}(f)\lambda(f) G(V_1V_2^2,Nf)}{|f|^{3/2} }\prod_{P |\Delta f} \Big(1-w^{\deg(P)}\Big)^{-1} u^{\deg(f)}\\
&\qquad\qquad= \mathcal{H}(V_1 ; u,w) \mathcal{I}(V_1V_2^2,N;u,w) \mathcal{J}(V_1V_2^2;u,w), 
\end{align*}
 where
$$ \mathcal{H}(V_1;u,w) =  \prod_{P \nmid V_1} \bigg( 1+ \frac{\chi_{C_1V_1}(P)\lambda(P) u^{\deg(P)}}{|P|}\Big(1-w^{\deg(P)}\Big)^{-1} \bigg), $$
$$ \mathcal{I}(V_1V_2^2,N;u,w) = \prod_{P | \Delta} \bigg( \sum_{j} \frac{\chi_{C_1}(P^j)\lambda(P^j) G(V_1V_2^2,P^{j+n_P}) u^{j \deg(P)}}{|P|^{3j/2}}   \bigg) \Big(1-w^{\deg(P)}\Big)^{-1}$$
and
\begin{align*}
\mathcal{J}(V_1V_2^2;u,w) &=  \prod_{\substack{P |V_1V_2\\P \nmid \Delta  }} \bigg( 1+ \sum_{j\geq1} \frac{\chi_{C_1}(P^j)\lambda(P^j) G(V_1V_2^2, P^j) u^{j \deg(P)}}{|P|^{3j/2}} \Big(1-w^{\deg(P)}\Big)^{-1}  \bigg)   \\
& \times \prod_{\substack{P \nmid V_1\\P | \Delta V_2 }}\bigg( 1+ \frac{\chi_{C_1V_1}(P)\lambda(P) u^{\deg(P)}}{|P|}\Big(1-w^{\deg(P)}\Big)^{-1} \bigg)^{-1}.
\end{align*}
We use the Perron formula for the sum over $f$ and obtain
\begin{align}\label{finalR<>0}
R(V\ne0) &= \frac{q^{2g+1}\overline{\tau(q)}}{|N|}\sum_{c_1+2c_2 \leq 2g+1} q^{-2c_2}\sum_{\substack{C_1 \in \mathcal{M}_{c_1} \\ C_1 | \Delta}} \frac{ \mu(C_1) \chi_{C_1}(N)}{|C_1|} \nonumber\\
& \times  \sum_{\substack{n\leq X \\ n +\deg(N) \text{ odd}}} \sum_{\substack{j\leq n+\deg(N)+c_1+2c_2-2g-2 \\ j+c_1\text{ odd}}} \sum_{V_1 \in \mathcal{H}_j} \sum_{V_2 \in \mathcal{M}_{(n+\deg(N)+c_1-j)/2+c_2-g-1}}\\
&\times  \frac{1}{(2 \pi i)^2}\oint_{|u|=r}\oint_{|w|=r} \mathcal{H}(V_1 ; u,w) \mathcal{I}(V_1V_2^2,N;u,w) \mathcal{J}(V_1V_2^2;u,w) \, \frac{du}{u^{n+1}}\frac{dw}{w^{c_2+1}}.\nonumber
\end{align}

 Let $r_1=q^{1/2-\varepsilon}$, $r_2=q^{-\varepsilon}$, and let $k_0$ be minimal such that $|r_1r_2^{k_0}| < 1$. Then we can write
\begin{equation}
\mathcal{H}(V_1;u,w) =  \mathcal{L}\Big(E\otimes\chi_{C_1V_1}, \frac uq\Big)\mathcal{L}\Big(E\otimes\chi_{C_1V_1}, \frac {uw}{q}\Big) \ldots\mathcal{L}\Big(E\otimes\chi_{C_1V_1}, \frac {uw^{k_0-1}}{q}\Big) \mathcal{K}(V_1;u,w), \label{expr}
\end{equation} where $$\mathcal{K}(V_1;u,w)\ll_\varepsilon |C_1|^\varepsilon$$ uniformly for $|u|\leq r_1$ and $|w|\leq r_2$. We also have $$ \mathcal{I} (V_1V_2^2,N;u,w)  \ll_\varepsilon  |V_1V_2|^{\varepsilon}\qquad\text{and}\qquad \mathcal{J}(V_1V_2^2;u,w)  \ll_\varepsilon  |V_2|^{\varepsilon}$$ in this region. We now move the contours in \eqref{finalR<>0} to $|u|=r_1$ and $|w|=r_2$. We then use the Lindel\"{o}f bound for each $L$-function and trivially bound the rest of the expression to obtain that
$$R(V\ne0)\ll_\varepsilon q^{X/2+\varepsilon g}.$$
Combining this with \eqref{main_th1} finishes the proof of Lemma \ref{lemma_th1}.
\end{proof}

To prove Theorem \ref{L^1}, note that from Lemma \ref{fe} we have
$$ \sum_{D \in \hstar}L ( E \otimes \chi_D, \tfrac{1}{2})= R_E ( 1, [ \mathfrak{n}/2 ] +2g+1)+ \epsilon_{2g+1} \epsilon(E) R_E( M, [  (\mathfrak{n}+1)/2 ] + 2g).$$
Using Lemma \ref{lemma_th1} shows that
$$ \frac{1}{ | \hstar|}\sum_{D \in \hstar} L ( E \otimes \chi_D, \tfrac{1}{2}) = c_1(M)L \big( \text{Sym}^2 E, 1\big) +O_\varepsilon\big(q^{-g+\varepsilon g}\big), $$ where
\begin{equation} \label{c1}
c_1(M) =\Big (\mathcal{C}_E(1;1) + \epsilon_{2g+1} \epsilon(E)\mathcal{C}_E(M;1)\Big)\prod_{P|\Delta}\frac{|P|+1}{|P|}.
\end{equation}
This finishes the proof of the theorem.

\section{Proof of Theorem \ref{L^2}}

Following Lemma \ref{fe}, for $X<2g$, we define
$$ \mathcal{M}_{1,E}(X): = (1+\epsilon)\sum_{f \in \mathcal{M}_{\leq X}} \frac{\lambda(f) \chi_D(f)}{\sqrt{|f|}},$$
so that
\begin{equation}\label{formulaL}
L ( E \otimes \chi_D, \tfrac{1}{2}) = \mathcal{M}_{1,E}(X) +\mathcal{E}_{1,E}(1,X,[\mathfrak{n}/2])+\epsilon_{2g+1}\epsilon(E)\,\mathcal{E}_{1,E}(M,X,[(\mathfrak{n}-1)/2]),
\end{equation}
where recall expression \eqref{truncation} for $\mathcal{E}_{1,E}(N,X,n)$.
Hence
\begin{align*}
L ( E \otimes \chi_D, \tfrac{1}{2}) ^2=&\, 2L ( E \otimes \chi_D, \tfrac{1}{2})\mathcal{M}_{1,E}(X) -\mathcal{M}_{1,E}(X) ^2\\
&\qquad\qquad+\Big(\mathcal{E}_{1,E}(1,X,[\mathfrak{n}/2])+\epsilon_{2g+1}\epsilon(E)\,\mathcal{E}_{1,E}(M,X,[(\mathfrak{n}-1)/2])\Big)^2.
\end{align*}
By Cauchy's inequality and Proposition \ref{prop2} we get
\begin{align*}
\sum_{D\in\mathcal{H}_{2g+1}^{*}}\, L ( E \otimes \chi_D, \tfrac{1}{2})^2=&\,2\sum_{D\in\mathcal{H}_{2g+1}^{*}}\, L ( E \otimes \chi_D, \tfrac{1}{2})\mathcal{M}_{1,E}(X)-\sum_{D\in\mathcal{H}_{2g+1}^{*}}\, \mathcal{M}_{1,E}(X)^2\\
&\qquad\qquad+O_{\varepsilon}\big( q^{2g}g^{1/2+\varepsilon}(2g-X)^2\big).
\end{align*}

Now using Lemma \ref{fe} again and expanding $\mathcal{M}_{1,E}(X)^2$, the first line of the equation above is 
\begin{align*}
&2\Big(S_{E}(1,[\mathfrak{n}/2]+2g+1,X)+S_{E}(1,[(\mathfrak{n}+1)/2]+2g,X)-S_{E}(1,X,X)\Big)\\
&\ +2\epsilon_{2g+1}\epsilon(E)\Big(S_{E}(M,[\mathfrak{n}/2]+2g+1,X)+S_{E}(M,[(\mathfrak{n}+1)/2]+2g,X)-S_{E}(M,X,X)\Big),
\end{align*}
where recall the definition of $S_E(M,X,Y)$ in Section \ref{mainpropS}.

Using Proposition \ref{prop1}, this is equal to
\begin{align*}
&|\hstar|c_2(M)\,L\big(\text{Sym}^2 E, 1\big)^3\, X+O(q^{2g})+O\big(q^{2g-X/5}g^3\big)+O\big(q^{5g/4+3X/8}g^{30}\big),
\end{align*}
where
\begin{equation}\label{C2}
c_2(M)=2\Big(\mathcal{C}_E(1;1,1,1)+\epsilon_{2g+1}\epsilon(E)\,\mathcal{C}_E(M;1,1,1)\Big)\prod_{P|\Delta}\frac{|P|+1}{|P|},
\end{equation}
and $\mathcal{C}_E(N;1,1,1)$ is defined in \eqref{Euler2}.
Thus
\begin{align*}
\frac{1}{|\hstar|}\sum_{D\in\mathcal{H}_{2g+1}^{*}}\, L ( E \otimes \chi_D, \tfrac{1}{2})^2=&\, c_2(M)\,L\big(\text{Sym}^2 E, 1\big)^3\, X+O\big(q^{-X/5}g^3\big)\\
&\qquad\qquad+O\big(q^{-3g/4+3X/8}g^{30}\big)+O_{\varepsilon}\big(g^{1/2+\varepsilon}(2g-X)^2\big).
\end{align*}
Choosing $X=2g-100\log g$ we obtain the theorem.

\section{Proof of Theorem \ref{L'L}}

Following Lemma \ref{fe'}, for $X<2g$ and fixed $n\in\mathbb{N}$, we define
$$ \mathcal{M}_{2,E}(X,n) := (1-\epsilon)\sum_{f \in \mathcal{M}_{\leq X}} \frac{\big(n+\deg(D)-\deg(f)\big)\lambda(f) \chi_D(f)}{\sqrt{|f|}},$$
so that
\begin{align}\label{formulaL'}
\epsilon^- L' ( E \otimes \chi_D, \tfrac{1}{2})& =(\log q) \mathcal{M}_{2,E}(X,[\mathfrak{n}/2]) \nonumber\\
&\qquad\qquad+(\log q)\Big(\mathcal{E}_{2,E}(1,X,[\mathfrak{n}/2])-\epsilon_{2g+1}\epsilon(E)\mathcal{E}_{2,E}(M,X,[\mathfrak{n}/2])\Big),
\end{align}
where recall that $\mathcal{E}_{2,E}(M,X,n)$ is given in \eqref{truncation2}.
Hence, by combining \eqref{formulaL} and \eqref{formulaL'},
\begin{align*}
&\epsilon_2^- L ( E_1 \otimes \chi_D, \tfrac{1}{2})L'( E_2 \otimes \chi_D, \tfrac{1}{2})\\
&\qquad\quad=\epsilon_2^- \mathcal{M}_{1,E_1}(X)L' ( E_2 \otimes \chi_D, \tfrac{1}{2})+(\log q)L( E_1 \otimes \chi_D, \tfrac{1}{2}) \mathcal{M}_{2,E_2}(X,[\mathfrak{n_2}/2])\\
&\qquad\qquad\quad\quad-(\log q)\mathcal{M}_{1,E_1}(X)\mathcal{M}_{2,E_2}(X,[\mathfrak{n_2}/2])\\
&\qquad\qquad\quad\quad+(\log q)\Big(\mathcal{E}_{1,E_1}(1,X,[\mathfrak{n_1}/2])+\epsilon_{2g+1}\epsilon(E_1)\,\mathcal{E}_{1,E_1}(M_1,X,[(\mathfrak{n_1}-1)/2])\Big)\\
&\qquad\qquad\qquad\times \Big(\mathcal{E}_{2,E_2}(1,X,[\mathfrak{n_2}/2])-\epsilon_{2g+1}\epsilon(E_2)\,\mathcal{E}_{2,E_2}(M_2,X,[\mathfrak{n_2}/2])\Big).
\end{align*}

 We bound the last term above using Cauchy's inequality and Proposition \ref{prop2}. In doing so we get
\begin{align}\label{second_mom}
&\sum_{D\in\mathcal{H}_{2g+1}^{*}}\, \epsilon_2^- L ( E_1 \otimes \chi_D, \tfrac{1}{2})L'( E_2 \otimes \chi_D, \tfrac{1}{2}) \nonumber \\
&\qquad=\sum_{D\in\mathcal{H}_{2g+1}^{*}}\, \epsilon_2^- \mathcal{M}_{1,E_1}(X)L' ( E_2 \otimes \chi_D, \tfrac{1}{2}) +(\log q)\sum_{D\in\mathcal{H}_{2g+1}^{*}}\, L( E_1 \otimes \chi_D, \tfrac{1}{2}) \mathcal{M}_{2,E_2}(X,[\mathfrak{n_2}/2])\nonumber \\
&\qquad\qquad-(\log q)\sum_{D\in\mathcal{H}_{2g+1}^{*}}\, \mathcal{M}_{1,E_1}(X)\mathcal{M}_{2,E_2}(X,[\mathfrak{n_2}/2])+O_{\varepsilon}\big( q^{2g}g^{1/2+\varepsilon}(2g-X)^3\big).
\end{align}
We shall estimate the remaining three terms using Proposition \ref{prop1}. They all have similar forms. For the first term, by Lemma \ref{fe'} again we have
\begin{align*}
&\sum_{D\in\mathcal{H}_{2g+1}^{*}}\, \epsilon_2^- \mathcal{M}_{1,E_1}(X)L' ( E_2 \otimes \chi_D, \tfrac{1}{2})\\
&\ =(\log q)\sum_{D\in\mathcal{H}_{2g+1}^{*}}\, (1+\epsilon_1)(1-\epsilon_2)\sum_{\substack{f\in\mathcal{M}_{\leq X}\\h \in \mathcal{M}_{\leq [\mathfrak{n_2}/2]+2g+1}\\}} \frac{\big([\mathfrak{n_2}/2]+2g+1-\deg(h)\big)\lambda_1(f)\lambda_2(h) \chi_D(fh)}{\sqrt{|fh|}}.
\end{align*}
By expanding out, this equals
\begin{align*}
&(\log q)\big([\mathfrak{n_2}/2]+2g+1\big)\Big(S_{E_1,E_2}(1,X,[\mathfrak{n_2}/2]+2g+1;0,0)\\
&\qquad\qquad\qquad\qquad\qquad\qquad+\epsilon_{2g+1}\epsilon(E_1)S_{E_1,E_2}(M_1,X,[\mathfrak{n_2}/2]+2g+1;0,0)\\
&\qquad\qquad\qquad\qquad\qquad\qquad-\epsilon_{2g+1}\epsilon(E_2)S_{E_1,E_2}(M_2,X,[\mathfrak{n_2}/2]+2g+1;0,0)\\
&\qquad\qquad\qquad\qquad\qquad\qquad-\epsilon(E_1)\epsilon(E_2)S_{E_1,E_2}(M_1M_2,X,[\mathfrak{n_2}/2]+2g+1;0,0)\Big)\\
&\qquad+\frac{\partial}{\partial \beta}\Big(S_{E_1,E_2}(1,X,[\mathfrak{n_2}/2]+2g+1;0,\beta)\\
&\qquad\qquad\qquad\qquad\qquad\qquad+\epsilon_{2g+1}\epsilon(E_1)S_{E_1,E_2}(M_1,X,[\mathfrak{n_2}/2]+2g+1;0,\beta)\\
&\qquad\qquad\qquad\qquad\qquad\qquad-\epsilon_{2g+1}\epsilon(E_2)S_{E_1,E_2}(M_2,X,[\mathfrak{n_2}/2]+2g+1;0,\beta)\\
&\qquad\qquad\qquad\qquad\qquad\qquad-\epsilon(E_1)\epsilon(E_2)S_{E_1,E_2}(M_1M_2,X,[\mathfrak{n_2}/2]+2g+1;0,\beta)\Big)\bigg|_{\beta=0},
\end{align*}
which is, by Proposition \ref{prop1} and Cauchy's residue theorem,
\begin{align*}
&|\hstar|\,c_3(M_1,M_2)L\big(\text{Sym}^2 E_1, 1\big)L\big( \text{Sym}^2 E_2,1\big)L\big ( E_1 \otimes E_2 , 1\big)g+O(q^{2g})+O\big(q^{5g/4+3X/8}g^{31}\big),
\end{align*}
where
\begin{align}\label{C3}
&c_3(M_1,M_2)=2(\log q)\Big(\mathcal{C}_{E_1,E_2}(1;1,1,1,0,0)+\epsilon_{2g+1}\epsilon(E_1)\,\mathcal{C}_{E_1,E_2}(M_1;1,1,1,0,0)\\
&\quad\quad-\epsilon_{2g+1}\epsilon(E_2)\,\mathcal{C}_{E_1,E_2}(M_2;1,1,1,0,0)-\epsilon(E_1)\epsilon(E_2)\,\mathcal{C}_{E_1,E_2}(M_1M_2;1,1,1,0,0)\Big)\prod_{P|\Delta}\frac{|P|+1}{|P|}.\nonumber
\end{align}
The other two terms in equation \eqref{second_mom} have the same asymptotics so we obtain
\begin{align*}
&\frac{1}{|\hstar|}\sum_{D\in\mathcal{H}_{2g+1}^{*}}\, \epsilon_2^- L ( E_1 \otimes \chi_D, \tfrac{1}{2})L'( E_2 \otimes \chi_D, \tfrac{1}{2})\\
&\qquad\qquad=c_3(M_1,M_2)L\big(\text{Sym}^2 E_1, 1\big)L\big( \text{Sym}^2 E_2,1\big)L\big ( E_1 \otimes E_2 , 1\big)g\\
&\qquad\qquad\qquad\qquad+O\big(q^{-3g/4+3X/8}g^{31}\big)+O_{\varepsilon}\big(g^{1/2+\varepsilon}(2g-X)^3\big).
\end{align*}
Choosing $X=2g-100\log g$ we obtain the theorem.

\section{Proof of Theorem \ref{L'^2}}

We argue as in the previous section. From \eqref{formulaL'} we have
\begin{align*}
&\epsilon_1^-\epsilon_2^- L '( E_1 \otimes \chi_D, \tfrac{1}{2})L'( E_2 \otimes \chi_D, \tfrac{1}{2})\\
&\qquad =(\log q)\Big(\epsilon_2^- \mathcal{M}_{2,E_1}(X,[\mathfrak{n_1}/2])L' ( E_2 \otimes \chi_D, \tfrac{1}{2})+\epsilon_1^-L'( E_1 \otimes \chi_D, \tfrac{1}{2}) \mathcal{M}_{2,E_2}(X,[\mathfrak{n_2}/2])\Big)\\
&\qquad\qquad-(\log q)^2\mathcal{M}_{2,E_1}(X,[\mathfrak{n_1}/2])\mathcal{M}_{2,E_2}(X,[\mathfrak{n_2}/2])\\
&\qquad\qquad+(\log q)^2\Big(\mathcal{E}_{2,E_1}(1,X,[\mathfrak{n_1}/2])-\epsilon_{2g+1}\epsilon(E_1)\,\mathcal{E}_{2,E_1}(M_1,X,[\mathfrak{n_1}/2])\Big)\\
&\qquad \quad \quad \times\Big(\mathcal{E}_{2,E_2}(1,X,[\mathfrak{n_2}/2])-\epsilon_{2g+1}\epsilon(E_2)\,\mathcal{E}_{2,E_2}(M_2,X,[\mathfrak{n_2}/2])\Big).
\end{align*}
 Bounding the last term above using Cauchy's inequality and Proposition \ref{prop2} leads to
\begin{align}\label{lastformula}
&\sum_{D\in\mathcal{H}_{2g+1}^{*}}\, \epsilon_1^-\epsilon_2^- L '( E_1 \otimes \chi_D, \tfrac{1}{2})L'( E_2 \otimes \chi_D, \tfrac{1}{2}) \nonumber \\
&\qquad=(\log q)\sum_{D\in\mathcal{H}_{2g+1}^{*}}\, \epsilon_2^- \mathcal{M}_{2,E_1}(X,[\mathfrak{n_1}/2])L' ( E_2 \otimes \chi_D, \tfrac{1}{2}) \nonumber \\
&\qquad\qquad+(\log q)\sum_{D\in\mathcal{H}_{2g+1}^{*}}\, \epsilon_1^-L'( E_1 \otimes \chi_D, \tfrac{1}{2}) \mathcal{M}_{2,E_2}(X,[\mathfrak{n_2}/2])\\
&\qquad\qquad-(\log q)^2\sum_{D\in\mathcal{H}_{2g+1}^{*}}\, \mathcal{M}_{2,E_1}(X,[\mathfrak{n_1}/2])\mathcal{M}_{2,E_2}(X,[\mathfrak{n_2}/2])+O_{\varepsilon}\big( q^{2g}g^{1/2+\varepsilon}(2g-X)^4\big).\nonumber 
\end{align}

We shall illustrate the evaluation of the third term using Proposition \ref{prop1}. The first two terms can be treated in the same way, and in fact they all have the same asymptotics. We have
\begin{align*}
&(\log q)^2\sum_{D\in\mathcal{H}_{2g+1}^{*}}\, \mathcal{M}_{2,E_1}(X,[\mathfrak{n_1}/2])\mathcal{M}_{2,E_2}(X,[\mathfrak{n_2}/2])=(\log q)^2\sum_{D\in\mathcal{H}_{2g+1}^{*}}\, (1-\epsilon_1)(1-\epsilon_2)\\
&\qquad\qquad\sum_{f,h\in\mathcal{M}_{\leq X}} \frac{\big([\mathfrak{n_1}/2]+2g+1-\deg(f)\big)\big([\mathfrak{n_2}/2]+2g+1-\deg(h)\big)\lambda_1(f)\lambda_2(h) \chi_D(fh)}{\sqrt{|fh|}}.
\end{align*}
By expanding out, this equals
\begin{align*}
&(\log q)^2\big([\mathfrak{n_1}/2]+2g+1\big)\big([\mathfrak{n_2}/2]+2g+1\big)\Big(S_{E_1,E_2}(1,X,X;0,0)-\epsilon_{2g+1}\epsilon(E_1)S_{E_1,E_2}(M_1,X,X;0,0)\\
&\qquad\qquad\qquad\qquad-\epsilon_{2g+1}\epsilon(E_2)S_{E_1,E_2}(M_2,X,X;0,0)+\epsilon(E_1)\epsilon(E_2)S_{E_1,E_2}(M_1M_2,X,X;0,0)\Big)\\
&\qquad+(\log q)\big([\mathfrak{n_1}/2]+2g+1\big)\frac{\partial}{\partial \beta}\Big(S_{E_1,E_2}(1,X,X;0,\beta)-\epsilon_{2g+1}\epsilon(E_1)S_{E_1,E_2}(M_1,X,X;0,\beta)\\
&\qquad\qquad-\epsilon_{2g+1}\epsilon(E_2)S_{E_1,E_2}(M_2,X,X;0,\beta)+\epsilon(E_1)\epsilon(E_2)S_{E_1,E_2}(M_1M_2,X,X;0,\beta)\Big)\bigg|_{\beta=0}\\
&\qquad+(\log q)\big([\mathfrak{n_2}/2]+2g+1\big)\frac{\partial}{\partial \alpha}\Big(S_{E_1,E_2}(1,X,X;\alpha,0)-\epsilon_{2g+1}\epsilon(E_1)S_{E_1,E_2}(M_1,X,X;\alpha,0)\\
&\qquad\qquad-\epsilon_{2g+1}\epsilon(E_2)S_{E_1,E_2}(M_2,X,X;\alpha,0)+\epsilon(E_1)\epsilon(E_2)S_{E_1,E_2}(M_1M_2,X,X;\alpha,0)\Big)\bigg|_{\alpha=0}\\
&\qquad+\frac{\partial^2}{\partial \alpha\partial \beta}\Big(S_{E_1,E_2}(1,X,X;\alpha,\beta)-\epsilon_{2g+1}\epsilon(E_1)S_{E_1,E_2}(M_1,X,X;\alpha,\beta)\\
&\qquad\qquad-\epsilon_{2g+1}\epsilon(E_2)S_{E_1,E_2}(M_2,X,X;\alpha,\beta)+\epsilon(E_1)\epsilon(E_2)S_{E_1,E_2}(M_1M_2,X,X;\alpha,\beta)\Big)\bigg|_{\alpha=\beta=0}.
\end{align*}
In view of Proposition \ref{prop1} and Cauchy's residue theorem, this is
\begin{align*}
&|\hstar|\,c_4(M_1,M_2)L\big(\text{Sym}^2 E_1, 1\big)L\big( \text{Sym}^2 E_2,1\big)L\big ( E_1 \otimes E_2 , 1\big)g^2+O(q^{2g}g)+O\big(q^{g/2+3X/4}g^{32}\big),
\end{align*}
where
\begin{align}\label{C4}
&c_4(M_1,M_2)=4(\log q)^2\Big(\mathcal{C}_{E_1,E_2}(1;1,1,1,0,0)-\epsilon_{2g+1}\epsilon(E_1)\,\mathcal{C}_{E_1,E_2}(M_1;1,1,1,0,0)\\
&\quad-\epsilon_{2g+1}\epsilon(E_2)\,\mathcal{C}_{E_1,E_2}(M_2;1,1,1,0,0)+\epsilon(E_1)\epsilon(E_2)\,\mathcal{C}_{E_1,E_2}(M_1M_2;1,1,1,0,0)\Big)\prod_{P|\Delta}\frac{|P|+1}{|P|}.\nonumber
\end{align}

The other two terms in equation \eqref{lastformula} have the same asymptotics so we obtain
\begin{align*}
&\frac{1}{|\hstar|}\sum_{D\in\mathcal{H}_{2g+1}^{*}}\, \epsilon_1^-\epsilon_2^- L '( E_1 \otimes \chi_D, \tfrac{1}{2})L'( E_2 \otimes \chi_D, \tfrac{1}{2})\\
&\qquad\qquad=c_4(M_1,M_2)L\big(\text{Sym}^2 E_1, 1\big)L\big( \text{Sym}^2 E_2,1\big)L\big ( E_1 \otimes E_2 , 1\big)g^2\\
&\qquad\qquad\qquad\qquad+O(g)+O\big(q^{-3g/2+3X/4}g^{32}\big)+O_{\varepsilon}\big(g^{1/2+\varepsilon}(2g-X)^4\big).
\end{align*}
Choosing $X=2g-100\log g$ we obtain the theorem.

\section{Proof of Corollary \ref{cor1}}

The results in Section \ref{upperbounds} imply that
\begin{equation}
\sum_{D \in \hstar}L(E\otimes\chi_D,\tfrac12)^4\ll_\varepsilon q^{2g}g^{6+\varepsilon}.
\label{_ub}
\end{equation}
We next obtain some upper bounds for moments of the derivatives. We have
\begin{align*}
L^{(l)}(E \otimes \chi_D, \tfrac{1}{2} )^k= \Big(\frac{l!}{2 \pi i} \Big)^k \oint \ldots \oint & \frac{L (E \otimes \chi_D, 1/2+\alpha_1)  \ldots  L(E \otimes \chi_D,1/2+\alpha_k)}{\alpha_1^{l+1} \ldots \alpha_k^{l+1}}d \alpha_1 \ldots d \alpha_k,
 \end{align*} where we are integrating along small circles of radii $r$ around the origin. Then using H\"{o}lder's inequality leads to
 \begin{align*}
\sum_{D \in \hstar}  \big|L^{(l)}(E \otimes \chi_D, \tfrac{1}{2})\big|^k \ll \frac{(l!)^k}{r^{lk+1}} \oint_{|\alpha|=r}\sum_{D \in \hstar} \big|L (E \otimes \chi_D, \tfrac{1}{2}+\alpha ) \big|^kd\alpha.
 \end{align*}
Choosing $r= 1/g$ and using upper bounds for moments of $L$--functions we get that
 \begin{align*}
\sum_{D \in \hstar} \big|L^{(l)}(E \otimes \chi_D, \tfrac{1}{2})\big|^k \ll_\varepsilon q^{2g} (l!)^k g^{lk+k(k-1)/2+\varepsilon}.
\end{align*}
In particular, with $l=1$ and $k=4$, we have
\begin{equation}
\sum_{D \in \hstar}L'(E\otimes\chi_D,\tfrac12)^4\ll_\varepsilon q^{2g}g^{10+\varepsilon}.
\label{der_ub}
\end{equation}

Now from H\"{o}lder's inequality we have
\begin{align*}
&\bigg(\sum_{\substack{D\in\mathcal{H}_{2g+1}\\(D,\Delta_1)=1}}L(E_1\otimes\chi_D,\tfrac12)^4\bigg)\bigg(\sum_{\substack{D\in\mathcal{H}_{2g+1}\\(D,\Delta_2)=1}}L'(E_2\otimes\chi_D,\tfrac12)^4\bigg)\Bigg(\sum_{\substack{D\in\mathcal{H}_{2g+1}\\(D,\Delta_1\Delta_2)=1\\\epsilon_2^-L(E_1\otimes\chi_D,1/2)L'(E_2\otimes\chi_D,1/2)\ne 0}}1\Bigg)^2\\
&\qquad\qquad\geq\bigg(\sum_{\substack{D\in\mathcal{H}_{2g+1}\\(D,\Delta_1\Delta_2)=1}}\epsilon_2^-L(E_1\otimes\chi_D,\tfrac12)L'(E_2\otimes\chi_D,\tfrac12)\bigg)^4.
\end{align*}
Combining \eqref{_ub} and \eqref{der_ub} with Theorem \ref{L'L} we get
\[
\#\big\{D\in\mathcal{H}_{2g+1}^*:\epsilon_2^-L(E_1\otimes\chi_D,\tfrac12)L'(E_2\otimes\chi_D,\tfrac12)\ne 0\big\}\gg_\varepsilon\frac{q^{2g}}{g^{6+\varepsilon}},
\]
which implies the first statement. The second statement can be obtained similarly.

 \bibliographystyle{amsalpha}

\bibliography{Bibliography_elliptic}

 \end{document}